\documentclass[10pt,english, a4paper,openany]{article}
\usepackage[latin1]{inputenc}
\usepackage[T1]{fontenc}
\usepackage{babel,textcomp}
\usepackage{amssymb}
\usepackage{amsmath}
\usepackage{amsfonts}
\usepackage{graphicx, url}
\usepackage{makeidx}
\usepackage{amsthm,color,enumerate}

\numberwithin{equation}{section} 
\numberwithin{figure}{section} 

\newtheorem{theorem}{Theorem}[section]

\newtheorem{proposition}[theorem]{Proposition}
\newtheorem{lemma}[theorem]{Lemma}
\newtheorem{example}[theorem]{Example}
\newtheorem{assumption}{Assumption}
\newtheorem{remark}{Remark}

\newcommand{\mc}{\mathcal}
\newcommand{\mb}{\mathbb}

\newcommand*\samethanks[1][\value{footnote}]{\footnotemark[#1]}

\DeclareMathOperator{\E}{\mathbb{E}} 

\title{Optimal control of systems with noisy memory and BSDEs with Malliavin derivatives}

\makeindex{}

\begin{document}


\author{K.R. Dahl\thanks{Department of Mathematics, University of Oslo, Pb. 1053 Blindern, 0316 Oslo, Norway. The research leading to these results has received funding from the European Research Council under the European Community's Seventh Framework Programme (FP7/2007-2013) / ERC grant agreement no. 228087.}
\and S.-E. A. Mohammed\thanks{Department of Mathematics, SIU-C, Carbondale, Illinois 62901, USA. Research supported in part by US NSF award DMS-0705970.}
\and B. {\O}ksendal\samethanks[1]
\and E. E. R{\o}se\samethanks[1]}

\date{21 August 2015}

\maketitle


\smallskip


\begin{abstract}
In this article we consider a stochastic optimal control problem where the dynamics of the state process, $X(t)$, is a controlled stochastic differential equation with jumps, delay and \emph{noisy memory}. The term noisy memory is, to the best of our knowledge, new. By this we mean that the dynamics of $X(t)$ depend on $\int_{t-\delta}^t X(s) dB(s)$ (where $B(t)$ is a Brownian motion). Hence, the dependence is noisy because of the Brownian motion, and it involves memory due to the influence from the previous values of the state process.

We derive necessary and sufficient maximum principles for this stochastic control problem in two different ways, resulting in two sets of maximum principles. The first set of maximum principles is derived using Malliavin calculus techniques, while the second set comes from reduction to a discrete delay optimal control problem, and application of previously known results by {\O}ksendal, Sulem and Zhang. The maximum principles also apply to the case where the controller has only partial information, in the sense that the admissible controls are adapted to a sub-$\sigma$-algebra of the natural filtration.
\end{abstract}
\section{Introduction}


In this article, we develop two approaches for analyzing optimal control for a new class of stochastic systems with noisy memory. The main objective is to derive necessary and sufficient criteria for  maximizing the performance functional on the underlying set of admissible controls. One should note the following unique features of the analysis:

\begin{itemize}
\item{} The state dynamics follows a controlled stochastic differential equation (SDE) driven by {\it noisy memory}: The evolution of the state $X$ at any time $t$ is dependent on its past history
$\int_{t-\delta}^t X(s) \, dB(s)$ where $\delta$ is the memory span and $dB$ is 
 white noise. In our opinion, it is reasonable and natural to consider this type of noisy dependence of the past.
\item{} The maximization problem is solved through a new backward stochastic differential equation (BSDE) that involves not only partial derivatives of the Hamiltonian but also their Malliavin derivatives.
\item{} Two independent approaches are adopted for deriving necessary and sufficient maximum principles for the stochastic control problem: The first approach is via Malliavin calculus and the second is a reduction of the dynamics to a two-dimensional controlled SDE with {\it discrete delay} and no noisy memory. In the second approach, the optimal control problem is then solved without resort to Malliavin calculus.
\item{} A natural link between the above two approaches is established as we show that a solution of the noisy memory BSDE can be obtained from a solution of the two-dimensional (time-) advanced BSDE (ABSDE) and vice versa.
\item To illustrate the usefulness of the Malliavin calculus approach, we outline in Section 8 an 
extension of the noisy memory problem where {\it the state dynamics cannot be reduced to a two-dimensional setting with discrete delay.}

\end{itemize}

To be somewhat more specific, we will outline below the scope of the results in the article. More precise regularity and measurability assumptions are provided in Sections 2,3 and 4.

The dynamics is described by the following one-dimensional controlled stochastic functional differential equation with {\it noisy memory}:

\begin{align}
  dX(t) &= b(t, X(t), Y(t), Z(t), \pi(t)) dt\nonumber\\
&\quad + \sigma(t, X(t), Y(t), Z(t), \pi(t))\, dB(t) \label{eq:StateEquation}\\
&\quad+\int_{\mb{R}} \gamma(t, X(t), Y(t), Z(t), \pi(t), \zeta) \tilde{N}(dt, d\zeta);  &t \in [0,T], \nonumber\\
X(t) &= \xi(t);  & t \in [-\delta, 0].\nonumber
\end{align}


\noindent In the above SDE, $\delta > 0$ is the memory span, $Y(t) := X(t - \delta)$ and the process
\begin{align}
\label{eq: 2.4}
Z(t) := \int_{t - \delta}^{t} X(s) \, d B(s)
\end{align}

\noindent stands for the {\it noisy memory} of the process $X$ at time $t$. The control process $\pi$  satisfies appropriate measurability and integrability requirements, while the random coefficients $b, \sigma, \gamma$ satisfy regularity and differentiability conditions. The dynamics is driven by a one-dimensional Brownian motion $B$,  a compensated Poisson random measure $\tilde{N}$ and an initial process $\xi$ on $[ -\delta, 0]$.

The main objective is to determine necessary and sufficient conditions for finding the maximizing control $\pi^*$ of the performance functional $J(\cdot)$, given by
\begin{align}
\label{eq: 2.6}
 J(\pi) := \E\Big[\int_0^T f(t,X(t),Y(t),Z(t),\pi(t)) dt + g(X(T))\Big],
\end{align}
\noindent for each admissible control process $\pi$. In the above relation, the symbol $\E[\cdot]$ denotes the expectation with respect to an underlying probability measure $P$ and $f, g$
 are given $C^1$ random functions satisfying appropriate measurability and integrability conditions.

In Section 2, we define the Hamiltonian associated with our maximal control problem together with a backward SDE (BSDE) ((2.19)-(2.21)). In Section 3, we obtain a sufficient maximum principle
(Theorem 3.1) which states that a solution of the BSDE yields an optimal control $\pi^*$ of the noisy memory control problem. This is achieved under sufficient Malliavin regularity and concavity conditions on the Hamiltonian and the performance functional. Under sufficient differentiability requirements on the underlying functions, we establish G\^ateaux-type differentiability for the performance functional $J$ (Lemma 4.4 and Theorem 4.5). This expresses the necessary condition for the optimal control problem in terms of the Hamiltonian
(Theorems 4.6 and 4.7).

In Section 5, we reduce the noisy memory dynamics to a $2D$ discrete delay format. By adapting the analysis in [13], we are able to establish necessary and sufficient conditions for solving the maximal control problem with noisy memory (Theorems 5.1, 5.2). A solution of the noisy memory BSDE is obtained using the solution of the $2D$ advanced BSDE (Theorem 6.1). 

In Section 7, an example with an optimal consumption problem is given, illustrating the two approaches to the maximal control problem.

In Section \ref{sec:generalization}, we show how the Malliavin calculus approach can be applied to more general noisy memory problems, where the reduction approach to the 2D dynamics is not feasible. In particular, we replace $Z(t)$ in (\eqref{eq:StateEquation} -\eqref{eq: 2.4}) by the more general noisy memory term 
\begin{align*}
 Z'(t):=\int_{t-\delta}^t \phi(t,s)X(s)dB(s).
\end{align*}


\section{The optimization problem}\label{sec: TheProblem}

In this section we formulate our main optimal control problem  for stochastic systems with noisy memory.

Let $B_t(\omega) = B(t, \omega); (t, \omega) \in [-\delta, \infty) \times \Omega$ be a Brownian motion and $\tilde{N}(dt, d\zeta) := N(dt, d\zeta) - \nu($d$\zeta)$d$t$ an independent compensated Poisson random measure, respectively, on a complete filtered probability space $(\Omega, \mc{F}, \{\mc{F}_t\}_{t \geq 0}, P)$. We assume that $\mb{F} := \{\mc{F}_t\}_{t \geq 0}$ is the filtration generated by $B$ and $\tilde{N}$ (augmented with the $P$-null sets)
and $\nu(d\zeta)$ is the L\'evy measure corresponding to the jump measure $N(dt,d\zeta)$. Let  $\mathbb G:=\{\mathcal{G}_t\}$ be a  sub filtration of $\mathbb F$, with $\mathcal G_t\subset \mathcal F_t$, and each $\mathcal G_t$ augmented with the $P$-null-sets. Note that no other conditions on $\mb{G}$ are required. In particular, our results hold for $\mc{G}_t = \mc{F}_0$ for all $t \geq 0$ (for example a deterministic control). We denote the set of admissible controls  by $\mathcal A_{\mathbb G}$. This set is contained in the set of all processes that are c\`{a}dl\`{a}g, in $L^2(\Omega\times[0,T])$, measurable wrt. the filtration $\mb{G}$ and take values in a subset $\mathcal V$ of $\mathbb R$.

Consider the following controlled stochastic differential equation (SDE) with discrete delay and {\it noisy memory}:
\begin{align}\label{eq:NoisyMemoryMainX}
\begin{split}
  dX(t) &= b(t, X(t), Y(t), Z(t), \pi(t)) dt\\
&\quad + \sigma(t, X(t), Y(t), Z(t), \pi(t)) d B(t) \\
&\quad+\int_{\mb{R}} \gamma(t, X(t), Y(t), Z(t), \pi(t), \zeta) \tilde{N}(dt, d\zeta); 
\end{split} &t \in [0,T], \\
X(t) &= \xi(t);  & t \in [-\delta, 0].
\end{align}
Here
\begin{align}\label{eq: 2.3}
 Y(t) := X(t - \delta)
\end{align}
\noindent where the positive constant $\delta$ is a discrete time-delay, while
\begin{align}\label{eq:NoisyMemoryMainZ}
Z(t) := \int_{t - \delta}^{t} X(s) d B(s)
\end{align}
\noindent represents the {\it noisy memory} of the process $X$ at time $t$. The process 
$\pi \in \mathcal A_{\mathbb G}$ is our control.
\begin{remark}
\label{remark: differentBM}
It is possible to have a different Brownian motion, say $\tilde B(t)$, driving the noisy memory process $Z(t)$ in 
(2.4).
In Sections \ref{sec: TheProblem}, \ref{sec: sufficient} and \ref{sec: NMP}, the only change would be that the Malliavin derivative $D_t$ with respect to $B$ should be replaced by the Malliavin derivative $\tilde D_t$ with respect to $\tilde B$ in (\ref{eq: 2.11}) and subsequent relations.  In Section \ref{sec: reduction}, everything still holds if the two Brownian motions are independent. If they are not independent, we can represent $\tilde B$ as a combination of $B$ and another independent Brownian motion $B_2$ as follows:
\begin{align*}
d\tilde B(t)=\alpha(t)dB(t)+\beta(t)dB_2(t),
\end{align*}
where $\alpha(t)=\frac{d}{dt}\E[\tilde B(t) B(t)]$ and $\alpha^2(t)+\beta^2(t)=1$. We omit the details.
\end{remark}

On the coefficient functions
\begin{align}b : \Omega\times[0,T] \times \mb{R} \times \mb{R} \times \mb{R} \times \mc{V} \rightarrow \mb{R},\\
\sigma : \Omega\times[0,T] \times \mb{R} \times \mb{R} \times \mb{R} \times \mc{V} \rightarrow \mb{R},\\
\gamma : \Omega\times[0,T] \times \mb{R} \times \mb{R} \times \mb{R} \times \mc{V} \times \mb{R} \rightarrow \mb{R},\end{align}
we impose the following set of assumptions
\begin{assumption}\label{ass:existenceUniqueness}$ $
\label{assumption: E}
\begin{enumerate}[i)]
\item\label{hyp:C1} The functions $ b(\omega, t,\cdot)$, $\sigma(\omega,t \cdot)$ and $\gamma(\omega,t,\zeta,\cdot)$
are assumed to be $C^1$ for each fixed $\omega,t, \zeta$, and  $\nabla$ denotes the gradients with respect to the variables $x,y,z,u$
\item \label{hyp:measurable}The functions  $b(\cdot,x,y,z,u)$ and $\sigma(\cdot, x,y,z,u)$, and $\gamma(\cdot, x,y,z,u,\zeta)$ are predictable
 for each $x,y,z,u$.
\item \emph{Lipschitz condition:} The functions $b,\sigma$
are Lipschitz continuous in the variables $x,y,z$, with the Lipschitz constant  independent of the variables $t,u, \omega$.
Also, there exists a function $\mathcal L\in L^2(\nu)$, independent of  $t,u, \omega$, such that
\begin{align}
 |\gamma(\omega,&t,x_1,y_1,z_1,u,\zeta)-\gamma(\omega,t,x_2,y_2,z_2,u,\zeta)|\\
&\leq \mathcal L(\zeta)\{|x_1-x_2|+|y_1-y_2|+|z_1-z_2| \}, \quad\nu-a.e. \zeta.
\end{align}
\item \emph{Linear growth:} The functions $b,\sigma, \gamma$ 
satisfy the linear growth condition in the variables $x,y,z$, with the linear growth  constant independent of the variables $t,u, \omega$
Also, there exists a non-negative function $\mathcal K\in L^2(\nu)$, independent of  $t,u, \omega$, such that
\begin{align}|\gamma(\omega,&t,x,y,z,u,\zeta)|\\
&\leq \mathcal K(\zeta)\{1+|x|+|y|+|z| \},\quad\nu-a.e. \zeta.
\end{align}
\end{enumerate}
\end{assumption}
Assumption  \ref{assumption: E} $\ref{hyp:C1})$ and Assumption \ref{assumption: E} $\ref{hyp:measurable})$ are sufficient to ensure the integrands  in equation (2.1) 
have  predictable versions, whenever $X$ is c\`{a}dl\`{a}g and adapted. It is always assumed that the $\tilde N$-integral is taken with respect to the predictable version of $\gamma(t, X(t), Y(t), Z(t), \pi(t), \zeta)$. Together with the Lipschitz and linear growth conditions, this ensures that for every $\pi\in\mathcal{A}_\mathbb{G}$,
there exists a unique c\`{a}dl\`{a}g adapted solution $X=X^{\pi}$ to the equation (2.1), 
satisfying
\begin{align}\label{ineq:XmomentEstimate}
\E[\sup_{t\in[-\delta,T]}|X(t)|^2]<\infty.
\end{align}
This can be seen, for example, by regarding equation (2.1) 
as a stochastic functional differential equation in
the sense of \cite{BCDDR} (cf. [9]).

The performance functional $J(\pi)$ of $\pi \in \mc{A}_{\mathbb G}$ is given by
\begin{align}\label{eq: performanceFunctional}
 J(\pi) := \E\Big[\int_0^T f(t,X(t),Y(t),Z(t),\pi(t)) dt + g(X(T))\Big],
\end{align}
\noindent where $\E[\cdot]$ denotes expectation with respect to $P$ and 
\begin{align*}f &:\Omega\times [0,T] \times \mb{R} \times \mb{R} \times \mb{R} \times \mc{V} \rightarrow \mb{R}&\textnormal{and}\\
 g &: \Omega \times \mb{R} \rightarrow \mb{R}
\end{align*} 
are given functions.
Throughout this paper, the functions $f,g$ are  assumed to satisfy the following conditions:
\begin{assumption}$ $
\label{assumption: F}
 \begin{enumerate}[i)]
\item  The functions $ f(\omega, t,\cdot)$ and $g(\omega,\cdot)$ are $C^1$  for each $t,\omega$.
 \item The functions $ f(\cdot,x,y,z)$ are progressively measurable, and $ g(\cdot,x,z)$ is $\mathcal F_T$ measurable.
\item Whenever $\pi\in\mathcal A_{\mathbb G}$, with corresponding $X(t) = X^{\pi}(t)$, $Y(t) = Y^{\pi}(t)$ and $Z(t) = Z^{\pi}(t)$,  it holds that \begin{align*}
\E\Big[\int_0^T (|f|+(\nabla f)^2)(t,X(t),Y(t),Z(t),\pi(t)) dt + (|g|+(g')^2)(X(T))\Big] < \infty.
\end{align*}
\end{enumerate}
\end{assumption}

The problem we will  consider is to find an optimal control $\pi^* \in \mc{A}_{\mathbb G}$ for $J(\cdot)$, i.e. to find $\pi^* \in \mc{A}_{\mathbb G}$ such that
\begin{align}
 \label{eq: 2.8}
\sup_{\pi \in \mc{A}_{\mathbb G}} J(\pi) = J(\pi^*).
\end{align}
To do so, we will require the following notion of the generalized Malliavin derivative for Brownian motion.

\subsection{The generalized Malliavin derivative for Brownian motion}
\label{subsec: GenMalliavin}

We refer to Nualart~\cite{Nualart}, Sanz-Sol\`{e}~\cite{SanzSole} and Di Nunno et al.~\cite{DOP09} for information about the Malliavin derivative $D_t$ for Brownian motion $B(t)$ and, more generally, L\'{e}vy processes. In Aase et al.~\cite{AaseEtAl}, $D_t$ was extended from the space $\mb{D}_{1,2}$ to $L^2(P)$, where $\mb{D}_{1,2}$ denotes the classical space of Malliavin differentiable $\mc{F}_T$-measurable random variables. The extension is such that for all $F \in L^2(\mc{F}_T, P)$, the following holds:

\begin{enumerate}
\item[$(i)$] $D_t F \in (\mc{S})^*$, where $(\mc{S})^* \supseteq L^2(P)$ denotes the Hida space of stochastic distributions,

\item[$(ii)$] the map $(t,\omega) \mapsto \E[D_t F | \mc{F}_t]$ belongs to $L^2(\mc{F}_T, \lambda \times P)$, where $\lambda$ denotes the Lebesgue measure on $[0,T]$.

Moreover, the following \emph{generalized Clark-Ocone theorem} holds:

\item[$(iii)$]
\begin{align}
F = \E[F] + \int_0^T \E[D_t F | \mc{F}_t] dB(t).\label{Clark-Ocone}
\end{align}
See \cite{AaseEtAl}, Theorem 3.11, and also \cite{DOP09}, Theorem 6.35.

\end{enumerate}
Notice that  by combining It\^o's isometry with the Clark-Ocone theorem, we obtain
\small
\begin{align}\label{eq:normOfDtF-varF}
 \E\Big[\int_0^T \E[D_t F|\mathcal F_t]^2 dt\Big]= \E\Big[\Big(\int_0^T \E[D_t F|\mathcal F_t]dB(t)\Big)^2\Big]=\E[(F^2-\E[F]^2)]
\end{align}\normalsize

As observed in Agram et al.~\cite{AgramOksendal}, we can also apply the Clark-Ocone theorem to show that:
\begin{proposition}(Generalized duality formula)
\label{prop: 2.1-NY}
Let $F \in L^2(\mc{F}_T, P)$ and let $\varphi(t) \in L^2(\lambda \times P)$ be adapted. Then

\begin{align}
\label{eq: gen-duality}
\E \Big[ F \int_0^T \varphi(t) dB(t) \Big] = \E \Big[ \int_0^T \E[D_t F | \mc{F}_t] \varphi(t) dt \Big]
\end{align}
\end{proposition}
\begin{proof}

By $(ii)$-$(iii)$ above and the It\^{o} isometry we have

\begin{align*}
\E \Big[ F \int_0^T \varphi(t) dB(t) \Big] &= \E \Big[ \Big(\E[F]+\int_0^T \E[D_t F | \mc{F}_t] dB(t)\Big) \Big(\int_0^T \varphi(t) dB(t)\Big)  \Big] \\[\smallskipamount]
&=\E \Big[ \Big(\int_0^T \E[D_t F | \mc{F}_t] dB(t)\Big) \Big(\int_0^T \varphi(t) dB(t)\Big)  \Big] \\[\smallskipamount]
                                &= \E \Big[ \int_0^T \E[D_t F | \mc{F}_t] \varphi(t) dt \Big].
\end{align*}

\end{proof}

For further results regarding the generalized Malliavin derivative, see {\O}ksendal and R{\o}se~\cite{OksendalRose}.

\subsection{The Hamiltonian and the associated BSDE}
\label{subsec: HamiltonianBSDE}

To solve problem~\eqref{eq: 2.8} we formulate a stochastic maximum principle, suitably modified for this situation:

First, define the Hamiltonian
\begin{align}
\mc{H} : [0,T] \times \mb{R} \times \mb{R} \times \mb{R} \times \mc{V} \times \mb{R} \times \mb{R} \times L^2(\nu) \rightarrow \mb{R}
\end{align}
by
\begin{align}
\mc{H}(t,x,y,z,u,p,q,r(\cdot)) &:= f(t,x,y,z,u) + b(t,x,y,z,u)p \nonumber\\
+ \sigma(t,x,y,&z,u)q + \int_{\mb{R}}\gamma(t,x,y,z,u,\zeta)r( \zeta)\nu(d\zeta)\label{Hamiltonian}
\end{align}
\noindent

Associated with the above Hamiltonian we have the following backward stochastic differential equation (BSDE) in the unknown processes $p, q$ and $r$:
\begin{align}
 dp(t) &= -\E[\mu(t) | \mc{F}_t] dt + q(t) dB(t) + \int_{\mb{R}} r(t,\zeta) \tilde{N}(dt, d\zeta); \hspace{0.5cm} 0 \leq t \leq T \nonumber\\
p(T) &= g'(X(T))\label{eq: 2.10}
\end{align}
\noindent where
\small\begin{align}
 \label{eq: 2.11}
\mu(t)= \frac{\partial{\mc{H}}}{\partial{x}}(t) + \frac{\partial{\mc{H}}}{\partial{y}}(t + \delta) \boldsymbol{1}_{[0,T-\delta]}(t)+ \int_{t}^{t+\delta} \mb{E} \Big[ D_t \big( \frac{\partial{\mc{H}}}{\partial{z}}(s) \big) | \mc{F}_t \Big] \boldsymbol{1}_{[0,T]}(s)ds.
\end{align}\normalsize
\noindent Here, $$\frac{\partial{\mc{H}}}{\partial{x}}(t)$$ is abbreviated notation for
\begin{align}\frac{\partial{\mc{H}}}{\partial{x}}(t, X(t), Y(t), Z(t), \pi(t), p(t), q(t), r(t,\cdot))\end{align} etc.


In particular, we say  the processes $p,q,r$ are  \emph{adjoint processes} corresponding to $\pi$
if the following holds: $p$ is  c\`adl\`ag and adapted, $q,r$ are predictable,
\begin{align}\label{est:adjointEquations}
 \E\big[\sup_{t\in [0,T]}p(t)^2+\int_0^T \bigg\{ q(t)^2 dt +\int_{\mathbb{R}}r(t,\zeta)^2\nu(d\zeta)+ \frac{\partial\mathcal H}{\partial z}(t)^2 \bigg \} dt\Big] <\infty,
\end{align}
 and the equalities (\ref{eq: 2.10}) holds $P$-a.s. for every $t \in [0,T]$.

\begin{remark}
 Note that due to the conditional expectation of the Malliavin derivative in the adjoint equation  \eqref{eq: 2.10} and the Clark-Ocone formula \eqref{Clark-Ocone}, the process $\mu$ has the alternative description
\begin{align*}
 \mu(t)= \frac{\partial{\mc{H}}}{\partial{x}}(t) + \frac{\partial{\mc{H}}}{\partial{y}}(t + \delta) \boldsymbol{1}_{[0,T-\delta]}(t)+ \int_{t}^{t+\delta}  \theta_s(t) \boldsymbol{1}_{[0,T]}(s)ds,
\end{align*}
where, for fixed $s$,  $\theta_s(t)$ is the unique process satisfying
\begin{align}
 \frac{\partial\mathcal H}{\partial z}(s)=\E\Big[ \frac{\partial\mathcal H}{\partial z}(s)\Big]+\int_0^s \theta_s(t) dB(t).
\end{align}
Although the proofs in Sections \ref{sec: sufficient}-\ref{sec: NMP} can  be carried out without resorting to Malliavin calculus, we have found the notation useful. We also remark that we have not been able to prove Theorem \ref{thm: NoisyMemBSDE} in Section \ref{sec: SolutionBSDE}, without using Malliavin calculus. Moreover, we emphasise that Malliavin calculus is needed as an efficient tool to actually \emph{find} this process $\theta_s(t)$.
See the example in Section 7.

\end{remark}

Note that the BSDE~(\ref{eq: 2.10}) is time-advanced in the sense that $\mu(t)$ involves future values like $X(t+\delta)$ etc. In this way the BSDE is similar to the time-advanced BSDE in  [13], 
 but note that the Malliavin derivative in the last term of (\ref{eq: 2.11}) constitutes a new ingredient. To the best of our knowledge, such BSDEs with Malliavin derivatives have not been studied before.

\subsection{Short-hand notation}\label{sec:ShortHandNotation}
Before we continue with the maximum principles, we introduce some abbreviated notation.
For any admissible control $\pi\in\mathcal{A}_{\mathbf G}$, we write $\mathbf X=(X,Y,Z)$ for the corresponding processes from the state equation (2.1) 
or $\mathbf X^\pi=(X^\pi, Y^\pi, Z^\pi)$, if confusion may occur. Similarily, adjoint processes corresponding to $\pi$ are denoted by $p,q,r$ or $p^\pi,q^\pi,r^\pi$.  Often, we will  mark a control with a diacritic. Then the corresponding processes will be marked with the same diacritic, i.e. the processes $\hat{\mathbf X}=\hat X, \hat Y, \hat Z$ and $\hat p, \hat q, \hat r$ corresponds to the control $\hat\pi$.

When any of the coefficient functions $b,\sigma, \gamma$, the utility function $f$, the Hamiltonian  $\mathcal H$ or any of their derivatives,  is evaluated in a set of processes all corresponding to the same control, we typically omit all variables except the time variable, and mark the function with the control or the diacritic when necessary.  As an example,  we write
\begin{align*}
\mathcal H(t) := \mathcal H^\pi(t)&:=\mathcal H(t, {\mathbf{X}}(t),{\pi}(t),{p}(t),{q}(t),{r}(t, \cdot))\\
\hat{\mathcal H}(t) &:= \mathcal H (t, \hat{\mathbf{X}}(t),\hat{\pi}(t),\hat{p}(t),\hat{q}(t),\hat{r}(t, \cdot)).
\end{align*}

\section{A sufficient maximum principle}
\label{sec: sufficient}
In this section we assume that the set $\mathcal V$ of all admissible controls is 
convex. Our main result here is a sufficient maximum principle for the system with noisy memory.

\begin{theorem}(Sufficient maximum principle for systems with noisy memory)\\
\label{thm: sufficient}
Let $\hat{\pi} \in \mc{A}_{\mathbb G}$ with corresponding $\hat{X}, \hat{Y}, \hat{Z}$, and adjoint processes $\hat{p}$, $\hat{q}, \hat{r}$.
%
%
%
Moreover, suppose that the following hold:
\begin{enumerate}[i)]
\item The functions
\begin{align}x \rightarrow g(x)\end{align}
and
\begin{align}
 \label{eq: 2.12}
(x,y,z,u) \rightarrow \mc{H}(t,x,y,z,u,\hat{p}(t),\hat{q}(t), \hat{r}(t,\cdot))
\end{align}
are concave a.s. for all $t \in [0,T]$.
\item

For every $v\in\mathcal V$
\begin{align} 
 \E\Big[\frac{\partial}{\partial u}\mathcal H\big( t, \hat {\mathbf X}(t), \hat\pi(t), \hat p(t), \hat q(t), \hat r(t)\big)\Big| \mathcal{G}_t \Big](v-\hat\pi(t))\leq 0\label{eq: 2.14}
\end{align} 
$dt\times P$-a.s.
\end{enumerate} 
Then $\hat{\pi}$ is an optimal control for the noisy memory control problem~(\ref{eq: 2.8}).
\end{theorem}


%
%
%
\begin{proof}
Fix  $\pi\in\mathcal{A}_{\mathbb{G}}$ with corresponding processes $X(t), b(t), \sigma(t), \gamma(t),$ $p(t), q(t), r(t)$. 

Write
\begin{align}
 \label{eq: 2.15}
J(\pi) - J(\hat{\pi}) = I_1 + I_2,
\end{align}
where
\begin{align}
 \label{eq: 2.16}
I_1 := \E[\int_0^T \Big( f(t,\mathbf{X}(t),\pi(t)) - f(t,\hat{\mathbf{X}}(t),\hat{\pi}(t)) \Big) dt]
\end{align}
and
\begin{align}
 \label{eq: 2.17}
I_2 := \E[g(X(T)) - g(\hat{X}(T))].
\end{align}
By the definition of $\mc{H}$ and  its concavity, 
we find that
\begin{align}
 \label{ineq:I_1}
I_1 &= \E\Big[ \int_0^T \Big\{\mc{H}(t,\mathbf{X}(t),\pi(t),\hat{p}(t),\hat{q}(t),\hat{r}(t, \cdot))-\mc{H}(t,\hat{\mathbf{X}}(t),\hat{\pi}(t),\hat{p}(t),\hat{q}(t),\hat{r}(t, \cdot))\nonumber\\
&\quad-\big(b(t,\mathbf{X}(t),\pi(t))-b(t,\hat{\mathbf{X}}(t),\hat{\pi}(t))\big)\hat{p}(t)\nonumber\\[\smallskipamount]
&\quad-\big(\sigma(t,\mathbf{X}(t),\pi(t))- \sigma(t,\hat{\mathbf{X}}(t),\hat{\pi}(t))\big)\hat{q}(t)\nonumber\\
&\quad-\int_{\mb{R}}\big(\gamma(t,\mathbf{X}(t),\pi(t), \zeta) - \gamma(t,\hat{\mathbf{X}}(t),\hat{\pi}(t), \zeta)\big)\hat{r}(t,\zeta)\nu(d \zeta)\Big\}dt\Big]\nonumber \\
&\leq \E\Big[\int_0^T\Big\{ \frac{\partial{\hat{\mc{H}}}}{\partial{x}}(t)\big(X(t) - \hat{X}(t)\big) + \frac{\partial{\hat{\mc{H}}}}{\partial{y}}(t)\big(Y(t) - \hat{Y}(t)\big)
+\frac{\partial{\hat{\mc{H}}}}{\partial{z}}(t)\big(Z(t) - \hat{Z}(t)\big)\nonumber \\
&\quad+ \frac{\partial{\hat{\mc{H}}}}{\partial u}(t)\big(\pi(t) - \hat{\pi}(t)\big) - \big(b(t) - \hat{b}(t)\big)\hat{p}(t) - \big(\sigma(t) - \hat{\sigma}(t)\big)\hat{q}(t) \nonumber\\
&\quad- \int_{\mb{R}}\big(\gamma(t,\zeta) - \hat{\gamma}(t,\zeta)\big)\hat{r}(t,\zeta)\nu(d \zeta)\Big\} dt\Big]
\end{align}
Since $g$ is concave and from the terminal condition of the adjoint equation, we have that
\begin{align}
 \label{eq: 2.19}
 I_2 &\leq \E[g'(\hat{X}(T))(X(T) - \hat{X}(T))] = \E[\hat{p}(T)(X(T) - \hat{X}(T))].
\end{align}
If we apply the It{\^o} formula to $\hat{p}(t)(X(t) - \hat{X}(t))$, we find that
\small\begin{align}
\hat p(T)(X(T)&-\hat X(T))=\int_0^T \E[-\hat\mu(t)|\mathcal F_t]\cdot\big(X(t)-\hat X(t)\big) +\hat p(t)\cdot \big(b(t)-\hat b(t)\big)\nonumber\\
&+\hat q(t)\cdot \big(\sigma(t)-\hat \sigma(t)\big)+ \int_{\mathbb{R}}\hat r(t,\zeta)\cdot \big(\gamma(t,\zeta)-\hat \gamma(t,\zeta)\big)\nu(d\zeta) dt\nonumber\\
&+\int_0^T\hat q(t)\cdot\big(X(t)-\hat X(t)\big)+ \hat p(t)\cdot \big(\sigma(t)-\hat \sigma(t)\big) dB(t)\label{eq:px_pxIto}\\
&+\int_0^T\int_{\mathbb{R}} \big[ \hat r(t,\zeta)\cdot\big(X(t)-\hat X(t)\big)\nonumber\\
&\quad\quad\quad\quad\quad\quad+ \big(\hat p(t)+\hat r(t,\zeta)\big)\cdot\big(\gamma(t,\zeta)-\hat\gamma(t,\zeta)\big) \big] \tilde N(dt,d\zeta).\nonumber
\end{align}\normalsize

Consider a suitable increasing sequence of stopping times $\tau_n$ defined by
\begin{align}
 \tau_n := T \wedge \inf \Big\{t > 0 :\mbox{ }  &\int_0^t \Big[ \Big( \hat q(s)\cdot\big(X(s)-\hat X(s)\big)+ \hat p(s)\cdot \big(\sigma(s)-\hat \sigma(s)\big) \Big)^2 \nonumber\\
&+  \int_{\mb{R}} \Big( \hat r(s,\zeta) \cdot \big( X(s)-\hat X(s)\big)
+ \big(\hat p(s)+\hat r(s,\zeta) \big) \nonumber\\
&\cdot \big( \gamma(s,\zeta)-\hat\gamma(s,\zeta) \big) \Big)^2 \nu(d\zeta)  \Big] ds\geq n\Big\}.
\end{align}\color{black}


\noindent It is easy to see that the sequence  $\{\tau_n \}_{n=1}^\infty$ converges to $T$. 
Now, since stochastic integrals with $L^2$-integrands have $0$ expectation, it follows that
\small\begin{align*}
\E[\hat p(\tau_n)(X(\tau_n)]&-\hat X(\tau_n)]=\E\Big[\int_0^{\tau_n} \E[-\hat\mu(t)|\mathcal F_t]\cdot\big(X(t)-\hat X(t)\big) +\hat p(t)\cdot \big(b(t)-\hat b(t)\big)\nonumber\\
&+\hat q(t)\cdot \big(\sigma(t)-\hat \sigma(t)\big)+ \int_{\mathbb{R}}\hat r(t,\zeta)\cdot \big(\gamma(t,\zeta)-\hat \gamma(t,\zeta)\big)\nu(d\zeta) dt.\Big]\nonumber
\end{align*}Note that the integrands are dominated by integrable processes, so we can pass to a limit. Combining this with \eqref{eq: 2.19}, we find that
\begin{align}\label{ineq:I_2}
I_2&\leq\E\Big[\int_0^T \E[-\hat\mu(t)|\mathcal F_t]\cdot\big(X(t)-\hat X(t)\big) +\hat p(t)\cdot \big(b(t)-\hat b(t)\big)\nonumber\\
&+\hat q(t)\cdot \big(\sigma(t)-\hat \sigma(t)\big)+ \int_{\mathbb{R}}\hat r(t,\zeta)\cdot \big(\gamma(t,\zeta)-\hat \gamma(t,\zeta)\big)\nu(d\zeta) dt\Big]. 
\end{align}
Finally, combining the estimates for $I_1$ and $I_2$ (\ref{ineq:I_1}, \ref{ineq:I_2}), we obtain
\small\begin{align}
J(\pi) - J(\hat{\pi}) &\leq \E\Big[\int_0^T \Big\{\frac{\partial{\hat{\mc{H}}}}{\partial{x}}(t)\cdot\big(X(t) - \hat{X}(t)\big) +  \frac{\partial{\hat{\mc{H}}}}{\partial{y}}(t)\cdot\big(Y(t) - \hat{Y}(t)\big) \nonumber\\
&+ \frac{\partial{\hat{\mc{H}}}}{\partial{z}}(t)\cdot\big(Z(t) - \hat{Z}(t)\big) + \frac{\partial{\hat{\mc{H}}}}{\partial u}(t)\cdot\big(\pi(t) - \hat{\pi}(t)\big) \\
&- \hat\mu(t)\cdot\big(X(t) - \hat{X}(t)\big)\Big\}dt\Big]\nonumber\\
&=\E\Big[\int_0^T \frac{\partial{\hat{\mc{H}}}}{\partial{y}}(t)\cdot\big(Y(t) - \hat{Y}(t)\big)dt\Big]\label{eq:Y1}\\
&\quad\quad- \E\Big[\int_0^T \frac{\partial{\hat{\mc{H}}}}{\partial{y}}(t + \delta)\cdot\big(X(t) - \hat{X}(t)\big)\boldsymbol{1}_{[0,T-\delta]}(t) dt\Big] \label{eq:Y2}\\
&+\E\Big[\int_0^T\frac{\partial{\hat{\mc{H}}}}{\partial{z}}(s)\cdot\big(Z(s) - \hat{Z}(s)\big)ds\Big] \label{eq:Z1}\\
&\quad\quad - \E[\int_0^T \int^{t + \delta}_t E[ D_t[\frac{\partial{\hat{\mc{H}}}}{\partial{z}}(s)] | \mc{F}_t] \boldsymbol{1}_{[0,T]}(s) (X(t) - \hat{X}(t)) ds dt\Big]\label{eq:Z2}\\
&+\E\Big[\int_0^T \frac{\partial{\hat{\mc{H}}}}{\partial u}(t)\cdot\big(\pi(t) - \hat{\pi}(t)\big)dt \Big]\nonumber\\
& =\E\Big[\int_0^T \frac{\partial{\hat{\mc{H}}}}{\partial u}(t)\cdot\big(\pi(t) - \hat{\pi}(t)\big)dt \Big].\label{eq:Jpi-Jpi}
\end{align}\normalsize
We will show  that the sum of the integrals (\ref{eq:Y1}-\ref{eq:Z2}) is in fact $0$.
Changing the order of integration and using the duality formula for Malliavin derivatives (Proposition~\ref{prop: 2.1-NY}), we get
\small\begin{align}\begin{split}
 \E\Big[\int_0^T& \frac{\partial{\hat{\mc{H}}}}{\partial{z}}(s)\cdot\big(Z(s) - \hat{Z}(s)\big)ds\Big] \\
&= \E\Big[\int_0^T \frac{\partial{\hat{\mc{H}}}}{\partial{z}}(s)\cdot \int_{s - \delta}^s \big(X(t) - \hat{X}(t)\big)dB(t) ds\Big] \\
&= \int_0^T \E\Big[\frac{\partial{\hat{\mc{H}}}}{\partial{z}}(s)\cdot \int_{s - \delta}^s \big(X(t) - \hat{X}(t)\big)dB(t)\Big] ds  \label{eq: 2.21}\\
&= \int_0^T \E[\int_{s - \delta}^s \E[ D_t(\frac{\partial{\hat{\mc{H}}}}{\partial{z}}(s)) | \mc{F}_t] \cdot\big(X(t) - \hat{X}(t)\big) dt] ds\\
&=  \E[\int_0^T \int^{t + \delta}_t \E[D_t(\frac{\partial{\hat{\mc{H}}}}{\partial{z}}(s)) | \mc{F}_t] \boldsymbol{1}_{[0,T]}(s) (X(t) - \hat{X}(t)) ds dt\Big]. 
\end{split}\end{align}\normalsize
Also, note that
\small\begin{align}
\E\Big[\int_0^T& \frac{\partial{\hat{\mc{H}}}}{\partial{y}}(t)\cdot\big(Y(t) - \hat{Y}(t)\big)dt\Big] \nonumber\\
&= \E\Big[\int_0^T \frac{\partial{\hat{\mc{H}}}}{\partial{y}}(t)\cdot\big(X(t - \delta) - \hat{X}(t - \delta)\big)dt\Big] \label{eq: 2.22}\\
&= \E\Big[\int_0^T \frac{\partial{\hat{\mc{H}}}}{\partial{y}}(t + \delta)\cdot\big(X(t) - \hat{X}(t)\big)\boldsymbol{1}_{[0,T-\delta]}(t) dt]\nonumber
\end{align}\normalsize
Now continuing where we left off from (\ref{eq:Jpi-Jpi}), we find that
\small\begin{align}
J(\pi)-J(\hat{\pi}) &\leq \E\Big[\int_0^T \frac{\partial{\hat{\mc{H}}}}{\partial u}(t)\cdot\big(\pi(t) - \hat{\pi}(t)\big)dt] \\[\smallskipamount]
&= \E[\int_0^T \E[\frac{\partial{\hat{\mc{H}}}}{\partial u}(t) | \mathcal{G}_t](\pi(t) - \hat{\pi}(t)) dt] \leq 0
\end{align}\normalsize
by (\ref{eq: 2.14}).
Hence, $\hat{\pi}$ is optimal.
\end{proof}

\section{A necessary maximum principle}
\label{sec: NMP}

Here we develop a Gateaux-type (or directional) differentiability property for the performance functional $J$ (Lemma 4.4, Theorem 4.5). The differentiability of $J$ is obtained under suitable regularity hypotheses on the
coefficients of the SDE with noisy memory, the performance functional and the set of admissible controls. See  Assumption \ref{assumption:NMP} below. The directional derivative of the performance functional yields
a necessary condition for the optimal control problem in terms of the Hamiltonian.

In the subsequent discussion, we will use the same notation $|\cdot|$ to denote any norm on $\mathbb R^n$, because such norms are all equivalent.


We impose the following set of assumptions throughout this section:
\begin{assumption}
\label{assumption:NMP}
\begin{enumerate}[i)]$ $
\item \label{assumption:derivativesBoundedAndLipshitz_equation}
The functions   $\nabla b$ and $ \nabla\sigma$  are bounded. The upper bound is denoted by $D_0$.
Also, there exists a non-negative function $D\in L^2(\nu)$ such that
\begin{align*}
 |\nabla\gamma(t,x,y,z,u,\omega, \zeta)|\leq D(\zeta)
\end{align*}
%
%
%
%
%

%
%
\item \label{assumption:derivativesBoundedAndLipshitz_pFunctional}
The functions $\nabla f$  and $  \nabla g$ are dominated by some \begin{align*}D_1(\cdot)\in L^2(\Omega\times[0,T]),&\quad\textnormal{and }D_2\in L^2(\Omega),\end{align*}
respectively.%
%

%
%
\end{enumerate}

\end{assumption}

\medskip

Let $\pi, \eta\in \mathcal{A}_{\mathbb G}$ and suppose $\eta$ is bounded. Consider the stochastic differential equation
\begin{align}
  d K(t)&= (K(t),  K(t-\delta), \int_{t-\delta}^t  K(r) dB(r),\eta(t))\cdot\Big[\nabla b(t,\mathbf X(t), \pi(t))\, dt \nonumber\\
& +\nabla \sigma(t,\mathbf X(t), \pi(t))\, dB(t)
+\int_{\mathbb{R}}\nabla\gamma (t,\mathbf X(t), \pi(t),\zeta) \tilde N(dt,d\zeta)\Big]^\mathsf{T} \label{eq:derivativeProcess}\\
 K(t)&=0,\quad t\in[-\delta,0].\nonumber
\end{align}We remark, that we regard the gradients as row vectors, and $\cdot$ as matrix multiplication.
\begin{lemma}\label{lemma:existenceAndUniqueness_K}

The equation
(\ref{eq:derivativeProcess}  ) has a unique c\`{a}dl\`{a}g solution \\$K=K^{\pi,\eta}\in L^2(\Omega\times[-\delta,T])$,  with
\begin{align}\label{ineq:ZmomentEstimate}
  \E[\sup_{t\in[-\delta,T]}|K(t)|^2]<\infty.
 \end{align}
\end{lemma}
The proof of the above lemma is straightforward, considering the equation (4.1) as a stochastic functional differential equation as in \cite{BCDDR}. The approach is similar to the one in  \cite{MR754561}, with the addition of applying Kunita's inequality for $\tilde N$-integrals (\cite{MR2090755}, Corollary 2.12). We remark  that the boundedness conditions on $\nabla b$, $\nabla \sigma$ and $\nabla \gamma$  are used in the proof.

To simplify the exposition in the rest of the section, we will adopt the following notation:
\small\begin{align}\mathbf K(t):&=\mathbf K^{\pi,\eta}(t):=\Big(K^{\pi,\eta}(t),  K^{\pi,\eta}(t-\delta), \int_{t-\delta}^t  K^{\pi,\eta}(s) \, dB_s\Big),&\textnormal{and}\nonumber\\
(\mathbf K(t), \eta(t)):&=(\mathbf K^{\pi,\eta}(t), \eta(t)):=\Big(K^{\pi,\eta}(t),  K^{\pi,\eta}(t-\delta), \int_{t-\delta}^t  K^{\pi,\eta}(s) \, dB_s, \eta(t)\Big)\label{not:K_eta},
\end{align}\normalsize
for $0 \leq t \leq T$.

\subsection{Directional differentiability of the performance functional}
Suppose now that $\pi,\eta\in \mathcal A_{\mathbb G}$. Also assume that there exist an interval $ I\subset \mathbb R$ containing $0$ such that the perturbations $\pi+s\eta$ is in $\mathcal A_{\mathbb G}$ for each $s\in  I$. The following lemmas give continuity and differentiability results for the function $$s\mapsto X^{\pi+s\eta}.$$
We begin by defining the random fields
\small\begin{align*}
 F_s(t):&=F^{\pi,\eta}_s(t):=X^{\pi+s\eta}(t)-X^{\pi}(t), \\
\mathbf F_s(t):&=\mathbf F^{\pi,\eta}(t):=\mathbf X^{\pi+s\eta}(t)-\mathbf X^{\pi}(t)=\Big(F^{\pi,\eta}_s(t), F^{\pi,\eta}_s(t-\delta), \int_{t-\delta}^t F^{\pi,\eta}_s(r)dB(r)\Big).
\end{align*}\normalsize
\begin{lemma}
\label{lem:F}There exists constants $C>0$, independent of $\pi,\eta$ such that
\begin{align}\label{ineq:continuitylemma}\E\Big[\sup_{0 \leq v\leq t}\Big|\mathbf F_s(v)|^2\Big]\leq C\parallel \eta\parallel^2_{L^2(\Omega\times[0,T])}s^2.
\end{align} Moreover there is measurable version of the random field $(\omega,t,s)\mapsto \mathbf F_s(t,\omega)$ such that for a.e. $\omega$,  $\mathbf F_s(t,\omega)\rightarrow 0$ as $s \rightarrow 0$ for each $t$.
\end{lemma}
\begin{proof}For simplicity, we consider the case where $b,\sigma=0$.
Define
\begin{align}
\beta_s(t):= \E\Big[\sup_{-\delta \leq v\leq t}\Big| F_s(v)|^2\Big].
\end{align}
Observe that by Kunita's inequality, it follows that 
\small\begin{align}
\begin{split}\label{ineq:boldBexpectation}
 \E[\sup_{0\leq v\leq t}&|\mathbf F_s(t)|^2]= \E\Big[\sup_{0\leq v\leq t}\Big\{|F_s(v)|^2+|F_s(v-\delta)|^2+\Big|\int_{v-\delta}^v F_s(r) dB(r)\Big|^2\Big\}\Big]\\
&\leq 2\beta_s(t)+\E\Big[\sup_{0\leq v\leq t}\Big|\int_{v-\delta}^v F_s(r) dB(r)\Big|^2\Big]\\
&\leq 2\beta_s(t)+C_{2,T}\int_{v-\delta}^v|F_s(r)|^2 dr\Big]\\
&\leq(2+\delta C_{2,T})\beta_s(t)
\end{split}
\end{align}\normalsize
Notice that since $\nabla \gamma$ is dominated by $D$, $\gamma$ is Lipschitz in all spacial variables, with Lipschitz constant $D(\zeta)$.
From the integral representation of $X$, It\^o's isometry, and finally the Lipschitz condition on $\gamma$ we find that
\small\begin{align*}
 \beta_s(t)&\leq \int_0^t\E\Big[\int_{\mathbb R}|\gamma(v,\mathbf X^{\pi+s\eta},\pi+s\eta)-\gamma(v,\mathbf X^{\pi},\pi)   |^2\nu(d\zeta)\Big]dv\\
&\leq \int_0^t\E\Big[\int_{\mathbb R}D(\zeta)^2|(\mathbf F_s(t), s\eta(t))|^2 \nu(d\zeta)\Big]dv\\
&\leq\parallel D\parallel^2_{L^2(\nu)}\int_0^t (2+\delta)\beta_s(v)\,dv +s^2\parallel\eta\parallel^2_{L^2(\Omega\times[0,T])}.
\end{align*}\normalsize
Hence by Gronwall's lemma there is a constant $C'>0$ such that
\begin{align}\label{est:beta_2}
 \beta_s(t)\leq C' s^2\parallel\eta\parallel^2_{L^2(\Omega\times[0,T])}.
\end{align}Combining this with the estimate (\ref{ineq:boldBexpectation}) yields the first part of the lemma. 

Now, using the first part of the lemma, and  an estimate similar to (\ref{ineq:boldBexpectation}), we find that for each $s_1, s_2\in I$
\begin{align*}
 \E[&\sup_{0\leq t\leq T}|\mathbf F^{\pi,\eta}_{s_1}(t)-\mathbf F^{\pi,\eta}_{s_2}(t)|^2]= \E[\sup_{0 \leq t\leq T}|\mathbf X^{\pi+s_1\eta}(t)-\mathbf X^{(\pi+s_1\eta)+(s_2-s_1)\eta}(t)|^2]\\
& \E[\sup_{0 \leq t\leq T}|\mathbf F^{\pi+s_1\eta,\eta}_{s_2-s_1}(t)|^2] \leq C|s_1-s_2|^2\parallel\eta\parallel^2_{L^2(\Omega\times[0,T])}.
\end{align*}Let $\mathcal D$ be the space of c\`adl\`ag paths from $[0,T]$ to $\mathbb R^3$ equipped with the uniform topology.
Then by the Kolmogorov-Totoki theorem (see e.g. \cite{MR2090755, Totoki}), it holds that the random field
\begin{align}
 I\times \Omega\ni(s,\omega)\mapsto F^{\pi,\eta}_s(\cdot,\omega)\in\mathcal D[0,T],
\end{align}  has a continuous version. Thus there is a version of $\mathbf F^{\pi,\eta}$ such that $(\omega,s,t)\mapsto F^{\pi,\eta}_s(t,\omega)$ is jointly measurable, c\`adl\`ag in $t$ and continuous in $s$. In particular for a.e. $\omega$ it holds that $\mathbb F^{\pi,\eta}_s(t,\omega)\rightarrow 0$ for every $t$, as $s\rightarrow 0$.
\end{proof}

Next, we define the random fields
\small\begin{align*}
 A_s(t):&=\frac{X^{\pi+s\eta}(t)-X^{\pi}(t)}{s}-K^{\pi,\eta}(t),  \quad -\delta \leq t \leq T;\\
 \mathbf A_s(t):&=\frac{\mathbf X^{\pi+s\eta}(t)-\mathbf X^{\pi}(t)}{s}-\mathbf K^{\pi,\eta}(t)=\Big(A_s(t), A_s(t-\delta), \int_{t-\delta}^t A_s(r)dB(r)\Big), 0 \leq t \leq T.
\end{align*}\normalsize

\begin{lemma}\label{lem:A}Suppose that $\pi,\eta\in\mathcal A_{}$. Then
 \begin{align}\label{lim:alpha}\E\Big[\sup_{0\leq v\leq t}\Big|\mathbf A_s(v)\Big|^2\Big]\rightarrow 0\end{align}
as $s\rightarrow 0$.
\end{lemma}
\begin{proof}
Define
\begin{align}
 \alpha_s(t):=
\E[\sup_{-\delta\leq v\leq t} |A_s(v)|^2]
\end{align}
Similarly as in the previous proof, we have
\small\begin{align}
\label{ineq:boldAexpectation}\begin{split}
 \E[\sup_{0\leq v\leq t}|\mathbf A_s(v)|^2]
\leq(2+C_{2,T}\delta)\alpha_s(t)\end{split}
\end{align}\normalsize

We remark that in order to use Taylor's formula for the $u$-variable, when $\mathcal{U}$ is not open, we need to assume that $b,\sigma,\gamma$ have $C^1$-extensions that are defined on an open set containing  $\mathcal{U}$. In extending the results to controls in e.g. $\mathbb R^n$, one needs to take extra care.

From the integral representation of $X$ and $K$, and by adding and subtracting a term, we find that
\small
\begin{align}
 A_s(t)&=\int_0^t\int_{\mathbb R_0} \frac{1}{s}\Big\{\gamma(v,\mathbf X^{\pi+s\eta},\pi+s\eta,\zeta)-\gamma(v,\mathbf{X}^{\pi},\pi,\zeta)\Big\}\nonumber\\
&\quad-\nabla \gamma(v,\mathbf X^{\pi},\pi,\zeta)\cdot(\mathbf K^{\pi,\eta}(v),\eta(v))^{\mathsf{T}}\tilde N(de,dv)\nonumber\\
\begin{split}\label{line:nablaGamma1}&=\int_0^t\int_{\mathbb R_0} \frac{1}{s}\Big\{\gamma(v,\mathbf X^{\pi+s\eta},\pi+s\eta,\zeta)-\gamma(v,\mathbf{X}^{\pi},\pi+s\eta,\zeta)\Big\}\\
&\quad-\nabla_{x,y,z}\gamma(v,\mathbf X^{\pi},\pi,\zeta)\cdot\mathbf K^{\pi,\eta}(v)^{\mathsf T}\tilde N(d\zeta,dv)\end{split}\\
\begin{split}\label{line:nablaGamma2}&+\int_0^t\int_{\mathbb R_0} \frac{1}{s}\Big\{\gamma(t,\mathbf X^{\pi},\pi+s\eta,\zeta)-\gamma(v,\mathbf{X}^{\pi},\pi,\zeta)\Big\}\\
&\quad-\frac{\partial}{\partial u}\gamma(v,\mathbf X^{\pi},\pi,\zeta)\eta(v)\tilde N(d\zeta,dv), \end{split}
\end{align}\normalsize
for $-\delta \leq t \leq T$.
From Kunita's inequality, we have that
\begin{align}
 \alpha_s(t)=\E[\sup_{-\delta\leq v\leq t}|A_s(v)|^2]\leq \int_0^t C_{2,T}2(I_{s,1}(v)+I_{s,2}(v))\, dv
\end{align}
where
\begin{align}
 \begin{split}\label{eq:I_s_1}I_{s,1}(v)&=\int_{\mathbb R_0}\E\Big[ \Big|\frac{1}{s}\Big\{\gamma(v,\mathbf X^{\pi+s\eta},\pi+s\eta,\zeta)-\gamma(v,\mathbf{X}^{\pi},\pi+s\eta,\zeta)\Big\}\\
&\quad-\nabla_{x,y,z}\gamma(v,\mathbf X^{\pi},\pi,\zeta)\cdot\mathbf K^{\pi,\eta}(v)^{\mathsf T}\Big|^2       \Big]\nu(d\zeta)\end{split}\\
\begin{split}\label{eq:I_s_2}I_{s,2}(v)&=\int_{\mathbb R_0}\E\Big[ \Big|\frac{1}{s}\Big\{\gamma(t,\mathbf X^{\pi},\pi+s\eta,\zeta)-\gamma(v,\mathbf{X}^{\pi},\pi,\zeta)\Big\}\\
&\quad-\frac{\partial}{\partial u}\gamma(v,\mathbf X^{\pi},\pi,\zeta)\eta(v) \Big|^2\Big]\nu(d\zeta).
\end{split}
\end{align}
We will show that 
$\int_0^t I_{s,2}(v)dv\rightarrow 0$ as $s\rightarrow 0$ and that  $I_{s,1}(v)$ are bounded by  terms on the form
\begin{align}\label{Gronwall}
\vartheta_s(v)+ \varphi(v)\alpha_s(v)                                                                    
\end{align}
where $\varphi\geq 0$ is integrable and for fixed $s$, $\vartheta_s\geq$ is integrable. Moreover it holds that  $\int_0^t \vartheta_s(v)dv\rightarrow 0$ as $s\rightarrow 0$. From Gr\"{o}nwall's inequality (see, e.g. the version in \cite{MR2512800}), it holds that
\begin{align*}
 \alpha_s(t)\leq 2C_{2,T}\int_0^T \Big(\vartheta_s(v)+I_{s,2}(v)\Big)dv\cdot \exp\Big\{2C_{2,T}\int_0^t    \varphi(v)dv\Big\}\rightarrow 0
\end{align*}
as $s\rightarrow 0$.
We first consider $I_{s,1}$ from equation (\ref{eq:I_s_1}). Let $\nabla_{x,y,z}$ denote the gradient with respect to the variables $x,y,z$. Applying Taylor's formula with integral remainder and adding and subtracting a term yields
\small\begin{align}
&I_{s,1}(v)=\int_{\mathbb R}\E\Big[\Big|   \frac{1}{s}\Big\{\gamma(v,\mathbf X^{\pi+s\eta},\pi+s\eta,\zeta)-\gamma(v,\mathbf{X}^{\pi},\pi+s\eta,\zeta)\Big\}\nonumber\\
&\quad-\nabla_{x,y,z}\gamma(v,\mathbf X^{\pi},\pi,\zeta)\cdot\mathbf K^{\pi,\eta}(v)^{\mathsf T}    \Big|^2\Big]\nu(d\zeta)\nonumber\\
&=\int_{\mathbb R}\E\Big[\Big|  \int_0^1\nabla_{x,y,z}\gamma(v,\mathbf X^{\pi}+\lambda\mathbf F_s(v), \pi+s\eta,\zeta)\cdot\frac{1}{s} \mathbf F_s(v)^{\mathsf T}\nonumber\\
&\quad-\nabla_{x,y,z}\gamma(v,\mathbf X^{\pi},\pi,\zeta)\cdot\mathbf K^{\pi,\eta}(v)^{\mathsf T}\,d\lambda  \Big|^2\Big]\nu(d\zeta) \nonumber\\
&=\int_{\mathbb R}\E\Big[\Big|    \int_0^1\nabla_{x,y,z}\gamma(v,\mathbf X^{\pi}+\lambda\mathbf F_s(v), \pi+s\eta,\zeta)\cdot \mathbf A_s(v)^{\mathsf T}\nonumber\\
&\quad+\Big(\nabla_{x,y,z}\gamma(v,\mathbf X^{\pi}+\lambda\mathbf F_s(v), \pi+s\eta,\zeta)-\nabla_{x,y,z}\gamma(v,\mathbf X^{\pi},\pi,\zeta)\Big)\cdot\mathbf K^{\pi,\eta}(v)^{\mathsf T}\,d\lambda    \Big|^2\Big]\nu(d\zeta)\nonumber\\
\begin{split}\label{lines:varphi_1}&\leq \int_{\mathbb R}\E\Big[2    \int_0^1\Big| \nabla_{x,y,z}\gamma(v,\mathbf X^{\pi}+\lambda\mathbf F_s(v), \pi+s\eta,\zeta)\cdot \mathbf A_s(v)^{\mathsf T}\Big|^2\,d\lambda    \Big]\nu(d\zeta)\end{split}\\
\begin{split}\label{lines:vartheta_1}&\quad+\int_{\mathbb R}\E \Big[   \int_0^1 2\Big|\Big(\nabla_{x,y,z}\gamma(v,\mathbf X^{\pi}+\lambda\mathbf F_s(v), \pi+s\eta,\zeta)\\
&\quad-\nabla_{x,y,z}\gamma(v,\mathbf X^{\pi},\pi,\zeta)\Big)\cdot\mathbf K^{\pi,\eta}(v)^{\mathsf T}\Big|^2\,d\lambda\Big]    \nu(d\zeta)\end{split}
\end{align}\normalsize
Now, we can use boundedness of $\nabla\gamma$ and the inequality (\ref{ineq:boldAexpectation}) to show that the term  (\ref{lines:varphi_1}), is bounded by
\begin{align*}2\parallel D\parallel^2_{L^2(\nu)}(2+\delta)\alpha_s(v). \end{align*}
Now consider the term (\ref{lines:vartheta_1}). Observe that since  $\lambda\mathbf F_s(v)$ converges pointwise to $0$, $\nabla_{x,y,z}\gamma$ is bounded and continuous, and the integrand is dominated by the $P\times[0,T]\times d\lambda\times\mathbb R$-integrable function $ 2D(\zeta)^2|\mathbf K(v)|^2$, it follows by Lebesgue's dominated convergence theorem that (\ref{lines:vartheta_1}) satisfies the conditions of $\vartheta_s$.


In a similar way, using Taylor's formula, we may show that
\small\begin{align*}
 &\int_0^T I_{s,2}(v)dv\\
&\leq \E\Big[\int_{\mathbb R}  \Big|\int_0^1\Big(\frac{\partial}{\partial u} \gamma(v,\mathbf X^{\pi},\pi+\lambda(s\eta),\zeta )&-\frac{\partial}{\partial u} \gamma(v,\mathbf X^{\pi},\pi,\zeta)\Big)\,d\lambda\cdot\eta(v)  \Big|^2     \nu(d\zeta)\Big]dv.
\end{align*}\normalsize
Now, the integrand is dominated by $ 2D(\zeta)^2\eta(s)^2$, which is\\ $P\times[0,T]\times d\lambda\times\mathbb R$-integrable, and converges point wise to $0$, because $\frac{\partial}{\partial u}\gamma$ is continuous. Therefore, 
\begin{align*}
 \int_0^T I_{s,2}(v)dv\rightarrow 0
\end{align*}
as $s\rightarrow 0$. This completes the proof of Lemma 4.3. \end{proof}
\begin{lemma}[\textbf{Differentiability of the performance functional $J$}]\label{lem:JDifferentiable1}
Suppose $\pi,\eta\in \mathcal A_{\mathbb G}$ with $\eta$ bounded. Suppose there exist an interval $ I\subset \mathbb R$ with $0\in I$, such that the perturbation $\pi+s\eta$ is in $\mathcal A_{\mathbb G}$ for each $s\in  I$.
Then the function $s\mapsto J(\pi+s\eta)$ has a (possibly one-sided) derivative at $0$ with
 \begin{align}
  \frac{d}{ds}J(\pi+s\eta)\Big|_{s=0}=\E\Big[ g'(X(T))\cdot K(T)+\int_0^T \nabla f(t,\mathbf{X}^{\pi}(t), \pi(t))\cdot (\mathbf K(t), \eta(t))^{\mathsf{T}} dt\Big].
\end{align}
\end{lemma}
\begin{proof}
For simplicity, we consider only the case where $g=0$. By using Taylor's formula with integral remainder, and proceeding as in the previous proof, one can show that
%
\small\begin{align}
 \Big|&\frac{J(\pi+s\eta)-J(\pi)}{s}-\E\Big[\int_0^T \nabla f(t,\mathbf{X}^{\pi},\pi)\cdot (\mathbf K(t),\eta(t))^{\mathsf{T}} dt\Big]\Big|\nonumber\\
&=\E\Big[\int_0^T\Big|\frac{f(t,\mathbf X^{\pi+s\eta},\pi(t)+s\eta)-f(t,\mathbf X^{\pi},\pi)}{s}-  \nabla f(t,\mathbf{X}^{\pi},\pi)\cdot (\mathbf K(t),\eta(t))^{\mathsf{T}} \Big|\Big]dt \nonumber\\
&\leq\E\Big[ \int_0^T\Big|\frac{f(t,\mathbf X^{\pi+s\eta},\pi+s\eta)-f(t,\mathbf X^{\pi},\pi+s\eta)}{s} -\nabla_{x,y,z}f(t,\mathbf X^{\pi},\pi)\cdot\mathbf K(t)^{\mathsf T}\nonumber \\
&+\Big|\frac{f(t,\mathbf X^{\pi},\pi+s\eta)-f(t,\mathbf X^{\pi},\pi)}{s} -\frac{\partial}{\partial u}f(t,\mathbf X^{\pi},\pi)\cdot\eta(t) dt\ \Big]\nonumber\\
&\leq  \E\Big[\int_0^T \int_0^1\Big|\nabla_{x,y,z} f(t,\mathbf X^{\pi}+\lambda \mathbf F_s, \pi+s\eta)\cdot\mathbf A_s(t)^{\mathsf T}\Big|d\lambda dt\Big] \label{line:varsigma3}\\
\begin{split}\label{line:varsigma}&+\E\Big[  \int_0^T   \int_0^1\Big| \Big(\nabla_{x,y,z}f(v,\mathbf X^{\pi}+\lambda\mathbf F_s, \pi+s\eta,\zeta)-\nabla_{x,y,z}f(v,\mathbf X^{\pi},\pi+\eta,\zeta)\Big)\\
 &\hspace{2.3cm}\cdot\mathbf K^{\pi,\eta}(v)^{\mathsf T}      \Big|\,d\lambda dt\Big]\end{split}\\
&+   \E \Big[ \int_0^T \Big|\int_0^1\Big(\frac{\partial}{\partial u} f(v,\mathbf X^{\pi},\pi+\lambda(s\eta),\zeta )-\frac{\partial}{\partial u} f(v,\mathbf X^{\pi},\pi,\zeta)\Big)\,d\lambda\cdot\eta(v)  \Big|dt\Big]\label{line:varsigma2}
\end{align}\normalsize 

The term (\ref{line:varsigma3}) tends to $0$ because from the boundedness of $\nabla_{x,y,z}f$ and Cauchy Schwartz inequality, we have 
\begin{align}
\E\Big[&\int_0^T \int_0^1\Big|\nabla_{x,y,z} f(t,\mathbf X^{\pi}+\lambda \mathbf F_s, \pi+s\eta)\cdot\mathbf A_s(t)^{\mathsf T}\Big|d\lambda dt\Big]\\
&\leq \Big(\E\Big[\int_0^T|D_1(t)|^2   \Big]dt\Big)^{\frac{1}{2}}   \Big(\E\Big[\int_0^T|\mathbf A_s(t)|^2   \Big]dt\Big)^{\frac{1}{2}},
\end{align}
and this tends to $0$ as $s \to 0$, by Lemma \ref{lem:A}

The term (\ref{line:varsigma}) tends to $0$ as $s \to 0$ because the integrand is dominated by the function $2D_1|\mathbf K|$ which is integrable, $\nabla_{x,y,z}f$ is continuous and $\lambda\mathbf F_s\rightarrow 0$ as $s \to 0$ for each $t,\lambda$ and a.e. $\omega$.

Similarly, the term (\ref{line:varsigma2}) tends to $0$ as $s \to 0$, because the integrand is dominated by the function $2D_1|\mathbf \eta|$ which is integrable, $\frac{\partial}{\partial u}f$ is continuous and $s\eta\lambda\rightarrow 0$, for each $t,\lambda$ and a.e. $\omega$. Hence the lemma is proved.\end{proof}

\begin{theorem}[\textbf{Differentiability of $J$ in terms of the Hamiltonian}]\label{lem:JDifferentiable2}
Suppose $\pi,\eta\in \mathcal A_{\mathbb G}$ with $\eta$ bounded. Suppose there exist an interval $ I\subset \mathbb R$ with $0\in I$ such that the perturbation $\pi+s\eta$ is in $\mathcal A_{\mathbb G}$ for each $s\in  I$.
Also assume that there exists unique corresponding adjoint processes
$p=p^{\pi}$   $q=q^{\pi}$ and $ r=r^{\pi}$.
Then
\begin{align}
\frac{d}{ds} J(\pi+s\eta)\big|_{s=0}=\E\Big[\int_0^T \frac{\partial}{\partial u}\mathcal H^{\pi}(t)\eta(t)dt\Big].
\end{align}
\end{theorem}

\begin{proof}
Define a sequence of stopping times by 
\small\begin{align*}
 \tau_n := T \wedge \inf \Big\{t > 0 : \int_0^t&\Big(|p(s)|^2+|q(s)|^2+\int_{\mathbb R}|r(s, \zeta)|^2\nu(d \zeta)\Big)\\
&\cdot \Big(|\mathbf K(s)|^2+|\eta(s)|^2 \Big) ds \geq n\Big\}.
\end{align*}\normalsize
Clearly $\tau_n\rightarrow T$ $P$-a.s. as $n\rightarrow \infty$.
Observe that
\begin{align}\label{eq:gZequalspZ}
\E[g'({X}(T))\cdot{K}(T)] =\E[p(T)K(T)].
\end{align} 
\normalsize
From It\^o's formula, we find that
\small\begin{align}\label{eq:p_K}\begin{split}
p({\tau_n}) &K({\tau_n})=\int_0^{\tau_n}p(t)(\mathbf K(t),\eta(t))\cdot \Big[\nabla b(t) dt  +\nabla \sigma(t) dB(t)\\
&+\int_{\mathbb{R}}\nabla\gamma (t,\zeta) \tilde N(dt,d\zeta)\Big]^{\mathsf{T}}\\
&+\int_0^{\tau_n} K(t)\Big[\E[-\mu(t,\pi)|\mathcal{F}_t]dt+q(t)dB(t)+ \int_{\mathbb{R}}r(t,\zeta)\tilde{N}(dt,d\zeta)    \Big]\\
&+\int_0^{\tau_n} q(t)(\mathbf K(t),\eta(t))\cdot\nabla\sigma(t)^{\mathsf{T}} dt\\
&+ \int_0^{\tau_n} \int_{\mathbb{R}}r(t,\zeta)(\mathbf K(t),\eta (t))\cdot\nabla\gamma(t ,\zeta)^{\mathsf{T}}\tilde N(dt,d\zeta)\\
&+\int_0^{\tau_n}\int_{\mathbb{R}}r(t,\zeta)(\mathbf K(t), \eta(t))\cdot\nabla\gamma(t,\zeta)^{\mathsf{T}}\nu(d\zeta)dt.\end{split}
\end{align}\normalsize
\noindent where $(\mathbf K(t),\eta(t))$ is the concatenation of $\mathbf K(t^-)$ and $\eta(t)$ (see (\ref{not:K_eta})).

The stochastic integrals  in (\ref{eq:p_K}) have zero expectation, since their integrands are square integrable by the definition of the stopping times. Combining this with the definition (\ref{Hamiltonian}) of the Hamiltonian yields
\small\begin{align*}
\E[p(\tau_n)K(\tau_n)]&=\E\Big[\int_0^{\tau_n}(\mathbf K(t),\eta (t))\cdot \Big(p(t)\nabla b(t)\\
&+q(t)\nabla \sigma(t)+\int_{\mathbb{R}}r(t,\zeta) \nabla \gamma(t,\zeta)\nu(d\zeta)    \Big)^{\mathsf{T}} dt\Big]\nonumber\\
&+\E\Big[\int_0^{\tau_n} K(t)\E[-\mu(t,\pi)|\mathcal{F}_t]dt\Big]\nonumber\\
&=\E\Big[\int_0^{\tau_n} \Big(\nabla  \mathcal H(t)-\nabla f(t)\Big)\cdot(\mathbf K(t),\eta (t))^\mathsf{T}\, dt\Big]\\
&+\E\Big[\int_0^{\tau_n} K(t)\E[-\mu(t,\pi)|\mathcal{F}_t]dt\Big].
\end{align*}\normalsize
Now, since the adjoint processes $K$ and $\eta$ are square integrable (see (4.2)), the integrands above are dominated by an integrable processes, and hence by the dominated convergence theorem, it follows that
\small\begin{align*}
 &\E[p(T)K(T)]=\lim_{n\rightarrow \infty}\E[p(\tau_n)K(\tau_n)]\\
&=\lim_{n\rightarrow \infty}\E\Big[\int_0^{\tau_n} \Big(\nabla  \mathcal H(t)-\nabla f(t)\Big)\cdot(\mathbf K(t),\eta (t))^\mathsf{T}\, + K(t)\E[-\mu(t,\pi)|\mathcal{F}_t]dt\Big]\\
&=\E\Big[\int_0^{T} \Big(\nabla  \mathcal H(t,\pi)-\nabla f(t,)\Big)\cdot(\mathbf K(t),\eta (t))^\mathsf{T}\, + K(t)\E[-\mu(t,\pi)|\mathcal{F}_t]dt\Big].
\end{align*}\normalsize
Then, using Lemma \ref{lem:JDifferentiable1} and (\ref{eq:gZequalspZ}) gives
\small
\begin{align}
  \frac{d}{ds}&J(\pi+s\eta)\Big|_{s=0}=\E\Big[ p(T)K(T)+\int_0^T \nabla f(t,\mathbf{X}^{\pi}(t), \pi(t))\cdot (\mathbf K(t), \eta(t))^{\mathsf{T}} dt\Big]\nonumber\\
  &=\E\Big[\int_0^T \nabla \mathcal \mathcal H(t,\mathbf{X}(t),\pi(t), p(t), q(t), r(t))\cdot(\mathbf K_{t},\eta_{t})^\mathsf{T} dt\Big]\nonumber\\
&+\E\Big[\int_0^T K(t)\E[-\mu(t,\pi)|\mathcal{F}_t]dt\Big]\nonumber\\
&=\E\Big[\int_0^T\frac{\partial }{\partial x}\mathcal H^\pi(t) K(t) dt\Big]- E\Big[\int_0^T K(t)\frac{\partial }{\partial x}\mathcal H^\pi(t)dt\Big]\nonumber\\
&+\E\Big[\int_0^T\frac{\partial }{\partial y}\mathcal H^\pi(t) K(t-\delta) \, dt\Big]- E\Big[\int_0^T K(t)\frac{\partial }{\partial y}\mathcal H^\pi(t+\delta)1_{[0,T-\delta]}(t)dt\Big]\label{eq:NMPlastLemma1}\\
\begin{split}\label{eq:NMPlastLemma2}&+ \E\Big[\int_0^T\frac{\partial }{\partial z}\mathcal H^\pi(t)\int_{t-\delta}^t K(r)dB(r)\Big]\\
&\quad\quad\quad\quad-\E\Big[\int_0^T K(t)\int_t^{t+\delta} \E \Big[ D_t\Big(\frac{\partial \mathcal H^\pi}{\partial z}(r)\Big) | \mc{F}_t \Big] 1_{[0,T]}(r)dr dt\Big]
\end{split}\\
&+ \E\Big[\int_0^T  \frac{\partial}{\partial u} \mathcal H^\pi(t)\eta(t)dt \Big]=\E\Big[\int_0^T  \frac{\partial}{\partial u} \mathcal H^\pi(t)\eta(t)dt \Big].\nonumber
\end{align}
\normalsize
To prove the last equality, it is sufficient to show that each of the lines (\ref{eq:NMPlastLemma1}) and  (\ref{eq:NMPlastLemma2}) is equal to zero. Observe first that
\begin{align*}
 \E\Big[\int_0^T& \frac{\partial }{\partial y}\mathcal H^\pi(t) K (t-\delta) dt\Big]=\E[\int_{\delta}^T\frac{\partial }{\partial y}\mathcal H^\pi(t) K(t-\delta) \, dt\Big]\\
&=\E\Big[\int_0^{T-\delta}\frac{\partial }{\partial y}\mathcal H^\pi(t+\delta) K(t) dt\Big]=\E\Big[\int_0^T  K(t)\frac{\partial }{\partial y}\mathcal H^\pi(t+\delta)1_{[0,T-\delta]}(t)  dt\Big].
\end{align*}
Also, using  Fubini's theorem and the duality formula for the Malliavin derivative (Proposition~\ref{prop: 2.1-NY}), we can show that:
\begin{align}
 \E\Big[\int_0^T& \frac{\partial }{\partial z}\mathcal H^\pi(t)\int_{t-\delta}^t K(r) dB(r) dt\Big]=\int_0^T\E\Big[\int_{t-\delta}^t \E \Big[ D_r \Big( \frac{\partial }{\partial z}\mathcal H^\pi(t)\Big) | \mc{F}_r \Big] K(r) dr  \Big]dt\nonumber\\
&=\E\Big[\int_0^T\int_{0}^T  K(r) \E \Big[ D_r \Big( \frac{\partial }{\partial z}\mathcal H^\pi(t)\Big) | \mc{F}_r \Big] 1_{[t-\delta,t]}(r) dtdr \Big]\nonumber\\
&=\E\Big[\int_0^T\int_{0}^T  K(r) \E \Big[ D_r \Big( \frac{\partial }{\partial z}\mathcal H^\pi(t)\Big) | \mc{F}_r \Big] 1_{[r,r+\delta]}(t) dtdr \Big]\label{eq:NMPlastLemma3}\\
&=\E\Big[\int_0^T\int_{r}^{r+\delta}  K(r) \E \Big[ D_r \Big( \frac{\partial }{\partial z}\mathcal H^\pi(t)\Big) | \mc{F}_r \Big] 1_{[0,T]}(t) dtdr \Big]\nonumber\\
&=\E \Big[ \int_0^T   K(t)\int_t^{t+\delta} \E [ D_t\Big(\frac{\partial \mathcal H^{\pi}}{\partial z}(r)\Big) | \mc{F}_t ] 1_{[0,T]}(r)dr  dt \Big].\nonumber
\end{align}
This completes the proof of the theorem.
\end{proof}
\subsection{Necessary maximum principles}


In this section, we develop necessary maximum principles in terms of the Hamiltonian.

\begin{theorem}[\textbf{Necessary maximum principle  I}]
\label{Thrm:NMP_I}
Suppose $\hat{\pi}\in \mathcal{A}_{\mathbb{G}}$. Denote by $\hat X$ the corresponding  state process and suppose that there exist corresponding adjoint processes $\hat p$, $\hat q$, and $\hat r$. 
In addition we assume that for each $t_0\in[0,T]$ and each bounded $\mathcal{G}_{t_0}$-measurable random variable $\alpha$, the process $\eta(t)=\alpha 1_{[t_0,T]}(t)$ belongs to $\mathcal{A}_{\mathbb G}$.
Then the following statements are equivalent
\begin{enumerate}[i)]
 \item\label{eq:ddsJ} For each bounded $\eta\in  \mathcal{A}_{\mathbb{G}}$, 
\begin{align*}
 \frac{d}{d s}J(\hat{\pi}+s\eta)\Big\rvert_{s=0}=0.
\end{align*}
\item\label{eq:nmp_i} For each $t\in[0,T]$,
\begin{align*}
 \E\bigg[ \frac{\partial }{\partial u}\mathcal H\big( t, \hat {\mathbf X}(t), \hat\pi(t), \hat p(t), \hat q(t), \hat r(t)\big)\Big| \mathcal{G}_t \bigg]=0\quad P\textnormal{-a.s.}
\end{align*}
\end{enumerate}
Suppose in addition that whenever $\eta\in \mathcal{A}_{\mathbb G}$ is bounded, there exists $\epsilon>0$ such that
    \begin{align*}
      \pi+s\eta\in \mathcal{A}_{\mathbb G}\quad\textnormal{for each }s\in (-\epsilon,\epsilon).
     \end{align*}
If $\hat\pi$ is optimal then \ref{eq:ddsJ}) and \ref{eq:nmp_i}) holds.
\end{theorem}
Using Theorem \ref{lem:JDifferentiable2}, the proof is  similar to that of Theorem 4.1 in \cite{OSZ1}.

If the space of admissible control values $\mathcal V$ is closed and  an optimal control have trajectories with values on the boundary of $\mathcal V$ on a non-negligible set, then the first necessary maximum principle is of little use. 

Suppose now  that $\mathcal A_{\mathbb G}$ is convex,  that $\hat\pi, \pi\in\mathcal A_{\mathbb G}$ with $\hat\pi$ optimal. Then the perturbation $\hat\pi+s(\pi-\hat\pi)\in\mathcal A_{\mathbb G}$ for every $s\in[0,1]$. And, hence it holds that
 $\frac{d}{d s}J(\hat{\pi}+s(\pi-\hat\pi))|_{s=0}\leq 0$,
for every $s\in[0,1]$, or equivalently (by Theorem  \ref{lem:JDifferentiable2}) that
\begin{align*}
\E\Big[\int_0^T \frac{\partial}{\partial u}\hat{\mathcal H}(t)(\pi(t)-\hat\pi(t))dt\Big]\leq 0.
\end{align*}

In particular this holds for every admissible $\pi$ of the form
\begin{align}\label{eq: piAssumption_NMP1d_v2}
 \pi_{h, t}(s):=
\begin{cases}
v,& s\in[t,t+h), \omega\in B \\
\hat\pi(s)&\textnormal{otherwise}
\end{cases}
\end{align}
where $t\in[0,T]$, $h>0$, $v\in\mathcal V$ and $B\in\mathcal G_{t}$. 
Fix $t\in[0,T]$, $B\in\mathcal{G}_{t}$.
Observe that
\begin{align*}
0&\geq\frac{1}{h}\E\Big[\int_{t}^{t+h} \frac{\partial}{\partial u}\hat{\mathcal H}(r)(\pi_{h,t}(r)-\hat\pi(r)) dr\Big]\\
&= \E\Big[ \frac{1}{h}\int_{t}^{t+h}\frac{\partial}{\partial u}\hat{\mathcal H}(r)(v-\hat\pi(r))dr 1_{B}  \Big]           
\end{align*}
Now since the above inequality holds for every $B\in \mathcal G_t$, it follows that
\begin{align*}
 0&\geq\E\Big[ \frac{1}{h}\int_{t}^{t+h}\frac{\partial}{\partial u}\hat{\mathcal H}(r)(v-\hat\pi(r))dr\big|\mathcal G_t \Big]   &&P-a.s.\\
&=\frac{1}{h}\int_{t}^{t+h}\E\Big[\frac{\partial}{\partial u}\hat{\mathcal H}(r)|\mathcal G_t \Big](v-\hat\pi(r)) dr  &&P-a.s.
\end{align*}
Letting $h\rightarrow 0$ in the above inequality, we obtain 
\begin{align*}
 \E\Big[\frac{\partial}{\partial u}\hat{\mathcal H}(r)|\mathcal G_t \Big](v-\hat\pi(r))\leq 0, &&P-a.s.
\end{align*}
for a.e. $t\in[0,T]$.

This gives the following maximum principle:

\begin{theorem}[\textbf{Necessary maximum principle  II}]\label{thrm:NMP_II}
Suppose that $\mathcal A_{\mathbb{G}}$ is a convex set, containing all controls of the form (\ref{eq: piAssumption_NMP1d_v2}). Assume that $\hat{\pi}\in \mathcal{A}_{\mathbb{G}}$ is optimal. Denote by $\hat X$ the solution of the corresponding  state equation and suppose  there exist corresponding adjoint processes $\hat p$, $\hat q$, and $\hat r$. 
Then
\begin{align*}
 \E\Big[\frac{\partial}{\partial u}\mathcal H\big( t, \hat {\mathbf X}(t), \hat\pi(t), \hat p(t), \hat q(t), \hat r(t)\big)\Big| \mathcal{G}_t \Big](v-\hat\pi(t))\leq 0
\end{align*} $dt\times P$-a.s.  
\end{theorem}

\section{Reduction of noisy memory to  discrete delay}
\label{sec: reduction}
In this section we formulate our one-dimensional noisy memory stochastic control problem as a two-dimensional control problem with {\it discrete delay}. This allows us to apply (a two-dimensional generalization of) previously known results from {\O}ksendal et al.~\cite{OSZ1} to get an alternative maximum principle for our original control problem. We then compare the maximum principles from the noisy memory-/Malliavin calculus approach and the discrete delay-approach.

Consider the original dynamics (2.1) 
for the process $X$, including the noisy memory term. For notational purposes, denote $X_1(t):=X(t)$. Define a new process $X_2(t)$ by
\begin{align}
\label{eq: x2}
X_2(t) := \int_{-\delta}^t X_1(s) dB(s).
\end{align}
Then, using the above transformation (\ref{eq: x2}), the dynamics in (2.1) 
can be rewritten as a two-dimensional SDE with {\it discrete delay} and {\it no noisy memory}:
\begin{align}
dX_1(t) &= b(t, X_1(t), X_1(t-\delta),X_2(t)-X_2(t - \delta),\pi(t))dt\nonumber \\[\smallskipamount]
&\quad+ \sigma(t, X_1(t), X_1(t-\delta),X_2(t)-X_2(t - \delta),\pi(t))dB(t) \nonumber\\[\smallskipamount]
&\quad+ \int_{\mb{R}} \gamma(t, X_1(t), X_1(t-\delta),X_2(t)-X_2(t - \delta),\pi(t))\tilde{N}(dt, d\zeta), \nonumber\\[\smallskipamount]
dX_2(t) &= X_1(t)dB(t), \nonumber\\[\smallskipamount]
X_1(t) &= \xi(t), \hspace{2cm} t \in [-\delta, 0], \nonumber\\[\smallskipamount]
X_2(t) &= \int_{-\delta}^t \xi(u)dB(u),  \hspace{0.35cm} t \in [-\delta, 0].\label{eq: 2dim}
\end{align}
In particular, we notice that by uniqueness of solutions, for any given $\pi\in\mathcal A_{\mathbb G}$, it follows that $X_1=X$, and that
\begin{align}\label{eq:X_Xtilde}
 \mathbf X(t)=(X_1(t),X_1(t-\delta), X_2(t)-X_2(t-\delta))
\end{align}
when $\mathbf X$ is defined as in Section (\ref{sec:ShortHandNotation}). Furthermore, under  Assumption \ref{ass:existenceUniqueness} in Section 2, a  unique solution always exists.
%
%
If we write $\tilde{X}(t):= ( X_1(t), X_2(t))^{\mathsf T}$ and $
 \tilde{Y}(t) := \tilde{X}(t-\delta)$, then the vector form of this equation is
\begin{align}
 d\tilde{X}(t) &= \tilde{b}(t, \tilde{X}(t), \tilde{Y}(t), \pi(t))dt + \tilde{\sigma}(t, \tilde{X}(t), \tilde{Y}(t), \pi(t))dB(t) \\ 
&\quad+ \int_{\mb{R}} \tilde{\gamma}(t, \tilde{X}(t), \tilde{Y}(t), \pi(t)) \tilde{N}(dt, d\zeta), 
\label{eqn: xtilde}
\end{align}
\noindent where
\begin{align}
 \tilde{X}(t) &:= \begin{bmatrix} \xi(t) \\ \int_{-\delta}^t \xi(l)dB_1(l) \end{bmatrix}, \hspace{0.2cm} t \in [-\delta,0], \\
\tilde{b}(t, x_1,x_2,y_1, y_2, \pi(t)) &:= \begin{bmatrix} b(t, x_1, y_1,x_2-y_2,u) \\ 0 \end{bmatrix}, \nonumber
\\
\tilde{\gamma}(t, x_1,x_2 y_1,y_2,u )&:= \begin{bmatrix} \gamma(t,  x_1, y_1,x_2-y_2,u) \\ 0 \end{bmatrix}, \nonumber
\\
\tilde{\sigma}(t,  x_1,x_2 y_1,y_2,u )& := \begin{bmatrix} \sigma(t, x_1, y_1,x_2-y_2,u) \\ x_1\end{bmatrix}.
\nonumber
\end{align}

%
%
This is a two-dimensional SDE with discrete delay and jumps. The results of {\O}ksendal et al.~\cite{OSZ1} can, in a straight-forward manner, be generalized to two dimensional dynamics. Hence, we can write down the performance function, Hamiltonian and adjoint equations as in \cite{OSZ1}.
The performance functional  (\ref{eq: performanceFunctional}), can be rewritten as 
\begin{align*}
 {J}(\pi) = \E\Big[\int_0^T \tilde{f}(t, \tilde{X}(t), \tilde{Y}(t), \pi(t)) dt + \tilde{g}(\tilde{X}(T))\Big], &&\pi \in \mc{A}_{\mb{G}},
\end{align*}
\noindent where
\begin{align*}
 \tilde{f}(t, x_1,x_2, y_1, y_2, u,) &= f(t, x_1, y_1,x_2-y_2,u) ,\textnormal{ and}\\
\tilde g(x_1, x_2)&=g(x_1).
\end{align*}
%

%

Now, the Hamiltonian for the reduced problem, denoted by $H$, is 
\begin{align}
H&(t, x_1, x_2, y_1, y_2, u, p_1, p_2, q_1, q_2, r_1(\cdot), r_2(\cdot))\nonumber\\
:&= \tilde f (t, x_1, x_2, y_1, y_2, u) +\tilde b^{\mathsf T}(t, x_1, x_2, y_1, y_2, u) \begin{bmatrix} p_1 \\ p_2 \end{bmatrix}\nonumber\\
&\quad+ \tilde \sigma^{\mathsf T}(t, x_1, x_2, y_1, y_2, u) \begin{bmatrix} q_1 \\ q_2 \end{bmatrix}+\int_{\mathbb R_0}\tilde \gamma^{\mathsf T}(t, x_1, x_2, y_1, y_2, u,\zeta) \begin{bmatrix} r_1(\zeta) \\ r_2(\zeta) \end{bmatrix} \nu(d\zeta)\nonumber\\
&=f(t,x_1, y_1, x_2-y_2,u)+b(t,x_1, y_1, x_2-y_2,u)p_1\nonumber\\
&\quad+\sigma(t,x_1, y_1, x_2-y_2,u)q_1+x_1q_2+ \int_{\mathbb R_0}  \gamma(t,x_1, y_1, x_2-y_2,u, \zeta)r_1(\zeta) \nu(d\zeta)\nonumber\\
&=\mathcal H(x_1, y_1, x_2-y_2, u, p_1, q_1, r_1)+x_1 q_2\label{eq:H_mathcalH}
\end{align}\normalsize
where $\mathcal H$ is the Hamiltonian from the 1-dimensional problem  (\ref{Hamiltonian}).

 The time-advanced BSDEs defining the adjoint equations for $\tilde p=(p_1, p_2)^{\mathsf T}$,  $\tilde q=(q_1, q_2)^{\mathsf T}$ and  $\tilde r=(r_1, r_2)^{\mathsf T}$ are given by the system

\small\begin{align*}
 d\tilde p(t)&=-\E[\nabla_x H^{\mathsf{T}}(t,\tilde X(t), \tilde Y(t),\pi(t), \tilde p(t), \tilde q(t), \tilde r(t))\\
&\quad +\nabla_y H^{\mathsf T}(t+\delta,\tilde X(t+\delta), \tilde Y(t+\delta),\pi(t+\delta), \tilde p(t+\delta), \tilde q(t+\delta), \tilde r(t+\delta))\boldsymbol 1_{[0,T]}(t+\delta)|\mathcal F_t] dt\\
&\quad \tilde q(t) dB(t)+ \int_{\mathbb R}\tilde r(t,\zeta)\tilde N(dt,d\zeta)\\
\tilde p(T)&=\nabla \tilde g^{\mathsf T}(\tilde X(T)).
\end{align*}\normalsize
If we write the equation for $p_1$ and $p_2$ separately, and combine this with (\ref{eq:X_Xtilde}) and (\ref{eq:H_mathcalH}), we obtain the following system
\begin{align}
\label{eq: AdjointAlt1}
\begin{split}
dp_1(t) &= -\E[\mu_1(t) | \mc{F}_t]dt + q_{1}(t)dB(t) + \int_{\mb{R}} r_1(t,\zeta) \tilde{N}(dt, d\zeta) \\
p_1(T) &= g'(X_1(T)),
\end{split}\\
&\nonumber\\
\label{eq: AdjointAlt2}
\begin{split}
dp_2(t) &= -\E[\mu_2(t) | \mc{F}_t]dt + q_{2}(t)dB(t) + \int_{\mb{R}} r_2(t,\zeta) \tilde{N}(dt, d\zeta) \\
p_2(T) &= 0.
\end{split}
\end{align}
where,
\begin{align}\label{mu_1}
\mu_1(t)&= q_2(t)
+\frac{\partial }{\partial x}\mathcal H(t, \mathbf X(t), \pi(t), p_1(t),  q_1(t), r_1(t) )\\
&\quad +\frac{\partial }{\partial y}\mathcal H(t+\delta, \mathbf X(t+\delta), \pi(t+\delta), p_1(t+\delta),  q_1(t+\delta), r_1(t+\delta) )\boldsymbol 1_{[0,T]}(t+\delta) \nonumber
\end{align}
\noindent and
\begin{align}\label{mu_2}
\mu_2(t)
&=\frac{\partial }{\partial z}\mathcal H(t, \mathbf X(t), \pi(t), p_1(t),  q_1(t), r_1(t) )\\
&\quad -\frac{\partial }{\partial z}\mathcal H(t+\delta, \mathbf X(t+\delta), \pi(t+\delta), p_1(t+\delta),  q_1(t+\delta), r_1(t+\delta) ) \boldsymbol 1_{[0,T]}(t+\delta).\nonumber
\end{align}
This is a 2-dimensional time advanced BSDE (ABSDE). In the 1-dimensional case, existence and uniqueness results for the solution of such ABSDEs can be found in {\O}ksendal, Sulem and Zhang~\cite{OSZ1}, Theorems 5.2-5.4. However, the extension to the 2-dimensional case is trivial, so the existence and uniqueness theorems apply to equations~\eqref{eq: AdjointAlt1} and \eqref{eq: AdjointAlt2} as well.

Now, we can state a sufficient maximum principle for this problem based on the (generalized) results from {\O}ksendal et al.~\cite{OSZ1}.  The following theorem holds under Assumption~\ref{ass:existenceUniqueness} of Section 2.
\begin{theorem} (A sufficient maximum principle via $2D$ discrete delay)\\
\label{thm: SufficientAlt}
Let $\hat{\pi} \in \mc{A}_{\mb{G}}$ with corresponding solution $\hat{X}_1, \hat{X}_2$ to the $2$-D discrete delay SDE~(\ref{eqn: xtilde}), with corresponding $\hat{Y}_1, \hat Y_2$. Suppose also that there exists corresponding adjoint processes $\hat{p}_1, \hat p_2, \hat q_1,  \hat{q}, \hat{r}_1$ and $\hat r_2$(i.e. solutions to the system (\ref{eq: AdjointAlt1})-(\ref{eq: AdjointAlt2}).) Suppose also that the following conditions hold:
\begin{enumerate}[i)]
 \item $(x_1,x_2) \mapsto g(x_1)$ and
\begin{align*}
 (x_1,x_2, y_1, y_2, u) \mapsto \mathcal H(t,x_1, y_1, x_2-y_2 ,u,\hat{p}_1(t),\hat{q}_1(t),\hat{r}_1(t))+x_1\hat q_2(t)
\end{align*}
\noindent are concave for all $t$ a.s.
\item
 \begin{align*}
 &\max_{v \in \mathcal U} \E[\mathcal H(t,\hat{\mathbf X}(t), v, \hat{p}_1(t), \hat{q}_1(t), \hat{r}_1(t,\cdot)) | \mc{G}_t] \\[\smallskipamount]
&= \E[\mathcal H(t, \hat{\mathbf X}(t),\hat\pi(t), \hat{p}_1(t), \hat{q}_1(t), \hat{r}_1(t,\cdot)) ) | \mc{G}_t]
\end{align*}
\noindent for all $t \in [0,T]$ a.s., where $\mc{U}$ is the set of admissible control values.
\end{enumerate}
Then $\hat{\pi}$ is an optimal control.
\end{theorem}
\begin{proof}
 This follows from the expressions above and a generalization of the results in {\O}ksendal et al.~\cite{OSZ1} using the stopping time technique from the proof of Theorem~\ref{thm: sufficient} (from {\O}ksendal and Sulem~\cite{OksendalSulemRiskMin}). Also, we have expressed the 2D- Hamiltonian $H$ in terms of our 1D Hamiltonian $\mathcal H$ as in (\ref{eq:H_mathcalH})
\end{proof}
%



Similarly, we can find a necessary maximum principle using the (generalized) results from {\O}ksendal et al.~\cite{OSZ1}. In the following theorem, we impose Assumption~\ref{ass:existenceUniqueness} of Section 2 and Assumption \ref{assumption:NMP} of Section 4.
\begin{theorem} (Necessary maximum principle via $2$D discrete delay)\\
\label{thm: NecessaryAlt}
Let $\hat{\pi} \in \mc{A}_{\mb{G}}$ with corresponding solution $\hat{X}_1, \hat{X}_2$ to the $2$D discrete delay SDE~(\ref{eqn: xtilde}), with corresponding $\hat{Y}_1, \hat Y_2$. Suppose also that there exists corresponding adjoint processes $\hat{p}_1, \hat p_2, \hat q_1,  \hat{q}, \hat{r}_1$ and $\hat r_2$ (i.e. solutions to the system (\ref{eq: AdjointAlt1})-(\ref{eq: AdjointAlt2}).)
%
%
%
Then, the following statements are equivalent,
 \begin{enumerate}
\item[$(i)$] For all bounded $\beta \in \mc{A}_{\mb{G}}$,
\begin{align}
 \frac{d}{ds} {J}(\hat{\pi} + s\beta) |_{s=0} = 0.
\end{align}

\item[$(ii)$] For all $t \in [0,T]$,
\begin{align}
 \E[\frac{\partial}{\partial u} \mathcal H(t,\hat{\mathbf X}(t), \hat\pi(t),\hat{p}_1(t), \hat{q}_1(t), \hat{r}_1(t, \cdot)) | \mc{G}_t] = 0 \mbox{ a.s.}
\end{align}
 \end{enumerate}
\end{theorem}
\begin{proof}
This follows from the expressions above and {\O}ksendal et al.~\cite{OSZ1}.
\end{proof}
\section{Solution of the noisy memory BSDE}
\label{sec: SolutionBSDE}
Now we have two pairs of necessary and sufficient maximum principles for the noisy memory problem. One pair of maximum principles, Theorem~\ref{thm: sufficient} and Theorem~\ref{Thrm:NMP_I}, was proved directly using Malliavin calculus. The other pair, Theorem~\ref{thm: SufficientAlt} and Theorem~\ref{thm: NecessaryAlt}, was proved indirectly by rewriting the problem as a $2$D optimal control problem with discrete delay and jumps, and then modifying previously known results of {\O}ksendal et al.~\cite{OSZ1} to derive the maximum principles. 

We have seen that essentially, the only difference in the 1D and the 2D maximum principles, is that in the  $1D$ maximum principle, the 1D Hamiltonian $\mathcal H$ is evaluated at the 1D adjoint processes $p,q,r$ and in the 2D maximum principle, the 1D Hamiltonian $\mathcal H$ is evaluated at the 2D adjoint processes $p_1, q_1, r_1$. 
This means that we do not actually need to know the processes $p_2, q_2$ and $r_2$, and thus the 2D approach seems unnecessarily complicated.


However, in this section, we establish a connection between
the adjoint processes in the Malliavin calculus approach to the corresponding ones in the
discrete delay approach. Recall that we say that processes $p,q,r$ is a solution to the \emph{noisy memory} BSDE if $r, q$ are predictable, the estimate \eqref{est:adjointEquations} holds and $p,q,r$ satisfy \eqref{eq: 2.10}-\eqref{eq: 2.11}. Similarily, we say that $p_i,q_1, r_i$, $i=1,2$ is a solution to the 2D time advanced BSDE  if $q_i, r_i, i=1,2$ are predictable, the estimate
\begin{align}
 \E\Big[\sup_{t\in[0,T]}|p_i(t)|^2+\int_0^T \bigg \{ |q_i(t)|^2 +\int_{\mathbb{R}} |r_i(t,\zeta)|^2\nu(d\zeta)\bigg \} \, dt\Big]<\infty
\end{align}
holds for $i=1,2$, and $q_i, r_i, i=1,2$ satisfy
\eqref{eq: AdjointAlt1} -\eqref{eq: AdjointAlt2}.

\begin{theorem}
\label{thm: NoisyMemBSDE}
(Solution of the noisy memory BSDE)\\
Suppose that $(p_i,q_i,r_i)$; $i=1,2$ is the solution of the $2$-dimensional ABSDE \eqref{eq: AdjointAlt1} -\eqref{eq: AdjointAlt2}.
Define $p(t) := p_1(t)$, $q(t) := q_1(t)$, $r(t,\zeta) := r_1(t,\zeta)$, and suppose that $\E[\int_0^T\frac{\partial\mathcal H}{\partial z}(t)^2 dt\Big]<\infty$. Then $(p,q,r)$ solves the \emph{noisy memory} BSDE \eqref{eq: 2.10}-\eqref{eq: 2.11}. Moreover,
\begin{align}
 \label{eq: (1)-notat}
q_2(t) = \int_t^{t+ \delta} \E \Big[ D_t(\frac{\partial \mc{H}}{\partial z}(s)) | \mc{F}_t \Big] \boldsymbol{1}_{[0,T]}(s) ds.
\end{align}
\end{theorem}
\begin{proof}
 For simplicity, we may assume $r = r_1 = r_2 =0$, since the jump terms do not play an essential role here. First note that in general we have that if $(p_2,q_2)$ solves a BSDE of the form
\begin{align}
 dp_2(t) = -\theta(t,p_2(t),q_2(t))dt + q_2(t)dB(t); \mbox{ } p_2(T) = F
\end{align}
\noindent then
\begin{align}
 \label{eq: (2)-notat}
q_2(t) = D_t p_2(t) \mbox{ } 
\end{align}
See  for example {\O}ksendal and R{\o}se~\cite{OksendalRose} for a proof in this general setting. For an earlier
proof valid under more restrictive conditions see e.g. Proposition 5.3 in El Karoui, Peng and Quenez~\cite{ElKarouiEtAl}

Also, note that the solution $p_2(t)$ of \eqref{eq: AdjointAlt2} can be written as
\small\begin{align}
 p_2(t) &= \E[\int_t^T \E[\mu_2(s) | \mc{F}_s]ds | \mc{F}_t]\nonumber \\
&= \int_t^T \E[\mu_2(s) | \mc{F}_t] ds \nonumber\\
&= \int_t^T \E[\frac{\partial \mc{H}}{\partial z}(s, \mathbf X(s), \pi(s), p_1(s), q_1(s), r_1(s))\\
&\quad - \frac{\partial \mc{H}}{\partial z}(s + \delta,\mathbf X(s+\delta), \pi(s+\delta), p_1(s+\delta), q_1(s+\delta), r_1(s+\delta)) \boldsymbol{1}_{[0,T-\delta]}(s) | \mc{F}_t]ds \nonumber\\
&= \int_t^{t+\delta} \E[\frac{\partial \mc{H}}{\partial z}(s, \mathbf X(s), \pi(s), p_1(s), q_1(s), r_1(s)) | \mc{F}_t]\boldsymbol 1_{[0, T]}(s)ds.  \nonumber
\end{align}\normalsize
Combining this with \eqref{eq: (2)-notat} and using proposition 3.12 in~\cite{DOP09}, we get \eqref{eq: (1)-notat}. Also
\begin{align*}
 q_2(t)&=D_t\int_t^{t+\delta} \E[\frac{\partial \mc{H}}{\partial z}(s, \mathbf X(s), \pi(s), p_1(s), q_1(s), r_1(s)) | \mc{F}_t]\boldsymbol 1_{[0, T]}(s)ds\\
&=\int_t^{t+\delta} \E[D_t\frac{\partial \mc{H}}{\partial z}(s, \mathbf X(s), \pi(s), p_1(s), q_1(s), r_1(s)) | \mc{F}_t]\boldsymbol 1_{[0, T]}(s)ds
\end{align*}
Now, by replacing $q_2$ with $$\int_t^{t+\delta} \E[D_t\frac{\partial \mc{H}}{\partial z}(s, \mathbf X(s), \pi(s), p_1(s), q_1(s), r_1(s)) | \mc{F}_t]\boldsymbol 1_{[0, T]}(s)ds,$$ in the definition (\ref{mu_1}) of $\mu_1$, we see that the solutions $p_1,q_1, r_1$ of (\ref{eq: AdjointAlt1}) solve the $1D$ adjoint equation (\ref{eq: 2.10}).\end{proof}
We also have a converse of this theorem:
\begin{theorem}
 Suppose that $p,q,r$ solves the `\emph{noisy memory}' BSDE \eqref{eq: 2.10}-\eqref{eq: 2.11}, and that $\E\big[\int_0^T \frac{\partial \mathcal H}{\partial z}(t)^2 dt \Big] <\infty$. Define $p_1:=p, q_1:=q, r_1:=r$ and
\begin{align}
 p_2(t)&:=\int_t^{t+\delta}\E\Big[\frac{\partial \mathcal H}{\partial z}(s) \Big|\mathcal F_t\Big]\boldsymbol1_{[0,T-\delta]}(s)ds\\
q_2(t)&:=\int_t^{t+\delta}\E\Big[D_t\frac{\partial \mathcal H}{\partial z}(s)\Big|\mathcal F_t\Big]\boldsymbol 1_{[0,T-\delta]}(s)ds\\
r_2&:=0.
\end{align}
 Then  $(p_i, q_i, r_i), i=1,2$ solves the $2$-dimensional ABSDE  \eqref{eq: AdjointAlt1} - \eqref{eq: AdjointAlt2}
\end{theorem}

\begin{proof}
Clearly,  the first part  \eqref{eq: AdjointAlt1} of the 2D ABSDE is satisfied. It remains to show that
\begin{align}\label{eq:Adjointpqr_2}
p_2(t)=\int_t^T\E\Big[  \frac{\partial \mathcal H}{\partial z}(s)-\frac{\partial \mathcal H}{\partial z}(s+\delta)\boldsymbol 1_{[0,T-\delta]}(s)\big|\mathcal F_s\Big]ds-\int_t^T q_2(s)dB(s).
\end{align}
From the relation \eqref{eq:normOfDtF-varF}, it follows that
\small\begin{align*}
\int_0^T&\E\Big[\int_0^T\E\Big[D_s\frac{\partial \mathcal H}{\partial z}(r)\big|\mathcal F_s\Big]^2 ds\Big]^{1/2}dr=\int_0^T\Big( \E\Big[\frac{\partial\mathcal H}{\partial z}(r)^2- \E\big[\frac{\partial\mathcal H}{\partial z}(r)\big]^2\Big] \Big)^{1/2}  dt<\infty.
\end{align*}\normalsize
Then, we can use the stochastic Fubini theorem (see e.g. \cite{Veraar}), and the following straightforward consequence of the Clark-Ocone Formula
\begin{align*}
 \frac{\partial \mathcal H}{\partial z}(r)=\E\Big[\frac{\partial \mathcal H}{\partial z}(r) \big|\mathcal F_t\Big]+\int_t^r \E\Big[D_s\frac{\partial \mathcal H}{\partial z}(r)\big|\mathcal F_s\Big] dB(s) &
\end{align*}
 to find that
\small\begin{align*}
 \int_t^T& q_2(s) dB(s)=\int_t^T \int_t^{T}\E\Big[D_s\frac{\partial \mathcal H}{\partial z}(r)\big|\mathcal F_s\Big]\boldsymbol 1_{[s,s+\delta]}(r)dr  dB(s)\\
&=\int_t^T\int_t^T \E\Big[D_s\frac{\partial \mathcal H}{\partial z}(r)\big|\mathcal F_s\Big]\boldsymbol 1_{[r-\delta,r]}(s) dB(r) ds\\
&=\int_t^{(t+\delta)\wedge T}\int_{t}^{r}\E\Big[D_s\frac{\partial \mathcal H}{\partial z}(r)\big|\mathcal F_s\Big] dB(s)dr+\int_{(t+\delta)\wedge T}^{T}\int_{r-\delta}^{r}\E\Big[D_s\frac{\partial \mathcal H}{\partial z}(r)\big|\mathcal F_s\Big] dB(s)dr\\
&=\int_t^{(t+\delta)\wedge T} \frac{\partial \mathcal H}{\partial z}(r)-\E\Big[\frac{\partial \mathcal H}{\partial z}(r) \big|\mathcal F_t\Big] dr+\int_{(t+\delta)\wedge T}^T\frac{\partial \mathcal H}{\partial z}(r)-\E\Big[\frac{\partial \mathcal H}{\partial z}(r)\big|\mathcal F_{r-\delta}\Big] dr\\
&=\int_t^T \frac{\partial \mathcal H}{\partial z}(r) dr-\int_{(t+\delta)\wedge T}^T\E\Big[\frac{\partial \mathcal H}{\partial z}(r) \big|\mathcal F_{r-\delta}\Big] dr -\int_t^{(t+\delta)\wedge T}\E\Big[\frac{\partial \mathcal H}{\partial z}(r) \big|\mathcal F_t\Big] dr\\
&=\int_t^T\E\Big[  \frac{\partial \mathcal H}{\partial z}(s)-\frac{\partial \mathcal H}{\partial z}(s+\delta)\boldsymbol 1_{[0,T-\delta]}(s)\Big|\mathcal F_s\Big]ds-p_2(t),
\end{align*}\normalsize
and hence \eqref{eq:Adjointpqr_2} holds.
\end{proof}

We notice that what this theorem  says, is that the 1D and the 2D maximum principles are essentially the same. However, we have two sets of adjoint equations, which can be an advantage in applications, as in general it can be extremely difficult to find solutions of both the 1D and the 2D time-advanced BSDE.

Also note that as a consequence of the theorem, a (unique) solution of the noisy memory BSDE exists whenever there exists a (unique) solution to the ABSDE~\eqref{eq: AdjointAlt1} and \eqref{eq: AdjointAlt2}. As mentioned, existence criteria for this ABSDE can be found in {\O}ksendal, Sulem and Zhang~\cite{OSZ1}.

\section{Application of the noisy memory maximum principle}
\label{sec: example}
As an example of the noisy memory optimal control problem, we consider two optimal consumption (optimal harvest) problem,

where the SDE for the state process $X(t)$ is given by
\begin{align}
dX(t) &= (a_0 Z(t) + a_1 X(t) - \pi(t)) dt + \sigma(t, X(t), Y(t), Z(t), \pi(t)) d B(t)  \nonumber\\[\smallskipamount]
&\quad
+\int_{\mb{R}} \gamma(t, X(t), Y(t), Z(t), \pi(t), \zeta) \tilde{N}(dt, d\zeta);  \hspace{1cm} t \in [0,T], \label{eq: exSDE}\\[\smallskipamount]
X(t) &= \xi(t);  \hspace{6.5cm} t \in [-\delta, 0].\nonumber
\end{align}
\noindent where $a_o, a_1 \in \mb{R}$, and $\sigma, \gamma$ are given functions satisfying the conditions of Section~\ref{sec: sufficient}. We say that $\pi\in\mathcal A_{\mathbb{G}}$ if 
$\pi>0$, $dP\times dt$ a.s.

Consider a performance functional $J(\pi)$ of $\pi \in \mc{A}_{\mb{G}}$ given by
\begin{align}
 J(\pi) := \E\Big[\int_0^T f(t,\pi(t)) dt + G(T)X(T)\Big],
\end{align}
where we assume that $f:\Omega\times[0,T] \times \mathcal V \rightarrow \mathbb R$ is concave with respect to $\pi$ for each $t$ and $\omega$. These assumptions are reasonable for a standard optimal consumption problem. 
We assume that $G(T)$ is a lognormal random variable (representing a stochastic terminal payoff price) of the form
\begin{equation} \label{eq7.3}
G(T) = \exp(\int_0^T \psi(t)dB(t))
\end{equation}
for some given deterministic function $\psi\in L^2([0,T])$.
We solve this noisy memory problem using the Malliavin stochastic maximum principle in Theorem~\ref{thm: sufficient}.

In this case, the  1D Hamiltonian is
\begin{align*}
&\mathcal H(t, x,y,z, \pi, p, q, r(\cdot)) \\
&= f(t,\pi) + (a_0 z + a_1 x - \pi)p + \sigma(t, x,y,z, u)q+ \int_{\mb{R}} \gamma(t, x,y,z, \zeta)r(\zeta)\nu(d\zeta),
\end{align*}
\noindent
The BSDE for the adjoint processes $p,q,r$ is given by equation~\eqref{eq: 2.10}, with
\small\begin{align*}
\mu(t) &= a_1 p(t) + \frac{\partial \sigma}{\partial x}(t)q(t) + \int_{\mb{R}} \frac{\partial \gamma}{\partial x}(t,\zeta)r(t,\zeta) \nu(\zeta) \\[\smallskipamount]
&+ \Big( \frac{\partial \sigma}{\partial y}(t+ \delta)q(t+ \delta) + \int_{\mb{R}} \frac{\partial \gamma}{\partial y}(t+ \delta,\zeta)r(t + \delta,\zeta) \nu(\zeta) \Big) \boldsymbol{1}_{[0,T-\delta]}(t) \\[\smallskipamount]
&+ \int_t^{t + \delta} \E \Big[ D_t \Big( a_0 p(t) + \frac{\partial \sigma}{\partial z}(s)q(s)  +  \int_{\mb{R}} \frac{\partial \gamma}{\partial z}(s,\zeta)r(s,\zeta) \nu(\zeta) \Big) \Big| \mc{F}_t \Big]  \boldsymbol{1}_{[0,T]}(s) ds.
\end{align*}\normalsize
\begin{example}[1D method]\label{example:Ft}
We now assume that
$$\sigma(t,X(t), Y(t), Z(t),\pi(t))=\sigma_0 (t) X(t),$$
where $\sigma_0(t)$ is a deterministic function and  $\gamma(t,\zeta)=0$ for all $t,\zeta$.  We also assume that $N=0$, so that $\{\mathcal F_t\}_{t \in [0,T]}$ is the natural filtration generated by the Brownian motion alone. Then the Hamiltonian is reduced to
 \begin{align*}
&\mathcal H(t, x,y,z, \pi, p, q, r(\cdot)) = f(t,\pi) + (a_0 z + a_1 x - \pi)p + \sigma_0(t)x q,
\end{align*}
and the adjoint equation takes the form:
\small
\begin{align}
\begin{cases}
 \label{eq7.4}
 dp(t) &= -\{ a_1 p(t) + \sigma_0(t) q(t) + a_0 \int_t^{t + \delta} \E \Big[ D_t \big( p(s)\big) | \mc{F}_t \Big] \boldsymbol{1}_{[0,T]}(s) ds \}dt + q(t) dB(t)\\[\smallskipamount]
p(T) &= G(T).
\end{cases}
\end{align}
\normalsize
Let us try a solution $p(t)$ of the form
\begin{align}\label{eq7.5}
p(t)&= p(0)\exp \big( \int_0^t \beta(s)dB(s)+ \int_0^t \{ \alpha(s) - \frac{1}{2}\beta^2(s) \}ds \big)\\
&=p(0) M(t) \exp \big( \int_0^t \alpha(s) ds \big)\nonumber\\
\end{align}
where $\alpha(s)$ and $\beta(s)$ are deterministic functions and $M(t)$ is the martingale
\begin{equation}
M(t) := \exp \big( \int_0^t \beta(s)dB(s)-\frac{1}{2} \int_0^t \beta^2(s) \}ds \big).
\end{equation}
Then by the chain rule for Malliavin derivatives we have
\begin{equation}
D_t p(s) = p(s) \beta(t) \boldsymbol{1}_{[t \leq s]}.
\end{equation}

Substituted into \eqref{eq7.4} this gives
\begin{align}\label{eq7.8}
\begin{cases}
dp(t) = \{ -A(t)p(t)+\sigma_0(t) q(t) \} dt + q(t) dB(t)\\
p(T)=G(T)
\end{cases}
\end{align}
where
\begin{equation}\label{eq7.9}
A(t) = a_1 + a_0 \beta(t) \int _t^{(t+\delta) \wedge T}\exp(\int_t^s \alpha(r)dr)ds.
\end{equation}
The solution of the linear BSDE \eqref{eq7.8} is (see e.g. Theorem 1.7 in \cite{OksendalSulemRiskMin})
\begin{equation}
p(t) = \frac{1}{\Gamma(t)}E[G(T)\Gamma(T) | \mathcal{F}_t],
\end{equation}
where
\begin{equation}
\begin{cases}
d\Gamma(t) = \Gamma(t) [ A(t) dt +\sigma_0(t) dB(t)]\\
\Gamma(0)=1,
\end{cases}
\end{equation}
i.e.
\begin{align}\label{eq7.12}
p(t)& = C \exp \big( \int_0^t \psi(s) dB(s) +\int_0^t \{ \frac{1}{2}\sigma_0(s)^2 -\frac{1}{2}(\sigma_0(s) + \psi(s))^2 -A(s)\}ds\big)
\end{align}
where 
\small
\begin{align}\label{eq7.13}
C= \exp \big( \int_0^T \{ A(s) - \frac{1}{2} \sigma_0(s)^2 + \frac{1}{2}(\sigma_0(s) + \psi(s))^2 \}ds \big).
\end{align}
\normalsize

Comparing \eqref{eq7.5} and \eqref{eq7.12}-\eqref{eq7.13} we see that if we choose
\begin{align}
\begin{cases}
\beta(t)=\psi(t) \\
\alpha(t) = \frac{1}{2}\sigma_0(t)^2 -\frac{1}{2}(\sigma_0(t) + \psi(s))^2 -A(t)\\
p(0)= C
\end{cases}
\end{align}
then $p(t)$ given by \eqref{eq7.5}, with the corresponding $q(t)=D_tp(t)=\beta(t)p(t)$ solve the BSDE \eqref{eq7.4}.
The first order (necessary) condition for the maximisation of the  Hamiltonian is
\begin{align}
\label{eq: FOC}
 \E\Big[ \frac{\partial f}{\partial \pi}(t,\pi(t)) | \mc{G}_t\Big] = \E\Big[p(t) | \mc{G}_t\Big].
\end{align}

Note that if $\mb{G}=\mb{F}$, then this reduces to
\begin{align}\label{eq:FOC2}
  \frac{\partial f}{\partial \pi}(t,\pi(t))= p(t).
\end{align}
since $f$, $\pi$ and $p$ are adapted to $\mb{F}$.\\

%
By the noisy memory  necessary maximum principle, Theorem~\ref{Thrm:NMP_I}, if $\pi^*$ is an optimal control, then $\pi^*$ solves~\eqref{eq: FOC}. Since  $f$ is concave, this is also a sufficient condition for optimality of $\pi^*$ by Theorem \ref{thm: sufficient}. 
\end{example}

Notice also the contribution of the noisy memory  term to the optimal solution: If we solve the same problem in Example \ref{example:Ft} without the memory term, i.e. where the SDE for the state process $X(t)$ is given by
\begin{align}\label{eq: exSDE_2}
\begin{split}
dX(t) &= (a_1 X(t) - \pi(t)) dt + \sigma_0(t) X(t)) d B(t)  
\hspace{1.7cm}t \in [0,T], \\
X(t) &= \xi(t);  \hspace{6.5cm} t \in [-\delta, 0].
\end{split}
\end{align}
(where $a_1 \in \mb{R}$, and $\sigma$  as above), the stochastic maximum principle implies that the first order condition for maximisation of the Hamiltonian is still given by
\begin{align}
\label{eq: FOC_2}
\frac{\partial f}{\partial \pi}(t,\pi(t)) = p(t).
\end{align}
However, in this case, the solution $p(t)$ of the adjoint BSDE given by the simpler expression
\begin{align}
 p(t) = \exp(\int_0^t \psi(s) dB(s) + \frac{1}{2} \int_t^T \psi^2(s) ds + a_1(T-t)),
\end{align}
which clearly is a different solution (actually, a special case) than in the example.

In  Example \ref{example:Ft}, we considered a class of optimization problems  where the dependence on the noisy memory gives us a different optimal solutuion than the corresponding equation without the noisy memory term (with the exception of the ``trivial`` case with $\psi=0$). 

Now, we give instead an example of a large class of problems where the dependence on the noisy memory term does not affect the closed form of the optimal solution.

\begin{example}[2D method]\label{example:SmallerGt}
 Consider the optimization problem given in the beginning of this section, this time with the only restriction being that $G(T)=1$.

 If we try to solve this stochastic control problem using the maximum principle from Section~\ref{sec: reduction}, we find that the 2D adjoint equations are given by
\begin{align}\label{eq: 2DAdjointExample}
\begin{split} dp_1(t)&=-\E[\mu_1(t)|\mathcal F_t]dt+q_1(t)d B(t)+ \int_{\mathbb R_0}r_1(t,\zeta)\tilde N(dt,d\zeta)\\
d p_2(t)&=-\E[\mu_2(t)|\mathcal F_t]dt+q_2(t)d B(t)+ \int_{\mathbb R_0}r_2(t,\zeta)\tilde N(dt,d\zeta),\\
p_1(T)&=1,\\
p_2(T)&=0,\end{split}
\end{align}where
\small\begin{align*}
\mu_1(t) &= a_1 p_1(t) + \frac{\partial \sigma}{\partial x}(t)q_1(t) + \int_{\mb{R}} \frac{\partial \gamma}{\partial x}(t,\zeta)r_1(t,\zeta) \nu(\zeta)+q_2(t) \\[\smallskipamount]
&+ \Big( \frac{\partial \sigma}{\partial y}(t+ \delta)q_1(t+ \delta) + \int_{\mb{R}} \frac{\partial \gamma}{\partial y}(t+ \delta,\zeta)r_1(t + \delta,\zeta) \nu(\zeta) \Big) \boldsymbol{1}_{[0,T-\delta]}(t),
\\ 
\mu_2(t) &= a_0 p_1(t) + \frac{\partial \sigma}{\partial z}(t)q_{1}(t) + \int_{\mb{R}} \frac{\partial \gamma}{\partial z} r_1(t,\zeta) \nu(d\zeta) \\[\smallskipamount]\nonumber
&- \boldsymbol{1}_{[0,T-\delta]}(t) \Big(a_0 p_1(t+\delta) + \frac{\partial \sigma}{\partial z}(t + \delta))q_{1}(t + \delta) 
+ \int_{\mb{R}} \frac{\partial \gamma}{\partial z}(t+\delta,\zeta)r_1(t+\delta,\zeta) \Big).
\end{align*}\normalsize
 One can easily verify that the processes $q_1=q_2=0$, $r_1=r_2=0$ and \begin{align}p_1(t)=e^{a_1(T-t)},&&p_2(t):=a_0\int_t^T p_1(s)-\boldsymbol 1_{[0,T-\delta]}(s) ds\end{align}
solve the equation (\ref{eq: 2DAdjointExample}). Since $p_1=p$, inserting the $2D$ adjoint processes into the $2D$ necessary maximum principle also yields the first order condition \eqref{eq: FOC}. A generalization of this problem, including a comparison of the $1D$ and $2D$ approaches, is given as an example in Section \ref{sec:generalization}.
\end{example}
\section{A generalized noisy memory control problem}\label{sec:generalization}
The Malliavin approach can be  extended to situations where the 2D approach is not applicable, e.g. if the \emph{noisy memory process} $Z(t)$ from \eqref{eq:NoisyMemoryMainX} -\eqref{eq:NoisyMemoryMainZ} is replaced by \emph{the generalized noisy memory process}
\begin{align*}
 Z'(t):=\int_{t-\delta}^t \phi(t,s)X(s)dB(s),
\end{align*}
where $\phi:\Omega\times[0,T]\times[-\delta,T]\rightarrow \mathbb R,$ is bounded, jointly measurable and with $\phi(\cdot, r,\cdot)$  adapted to $\{\mathcal F_t\}_{t\in[r-\delta,r]}$ for each fixed $r\in[0,T]$. Such a state equation can not in general be reduced to a 2D discrete delay equation, due to the dependence on $t$ in $\phi$. 

If we also replace
 $\mu$ in the adjoint equation \eqref{eq: 2.10}-\eqref{eq: 2.11} by
\begin{align}\begin{split}
 \mu'(t)&= \frac{\partial{\mc{H}}}{\partial{x}}(t) + \frac{\partial{\mc{H}}}{\partial{y}}(t + \delta) \boldsymbol{1}_{[0,T-\delta]}(t)\\
&\quad+ \int_{t}^{t+\delta} \mb{E} \Big[ D_t \big( \frac{\partial{\mc{H}}}{\partial{z}}(s) \big) | \mc{F}_t \Big]\phi(t,s) \boldsymbol{1}_{[0,T]}(s)ds, \end{split}\tag{\ref{eq: 2.11}'},
\end{align}(and otherwise leave the set-up exactly as in Section \ref{sec: TheProblem}),
it is fairly straightforward to show that Theorems \ref{thm: sufficient}, \ref{Thrm:NMP_I}, \ref{thrm:NMP_II} are valid also for this \emph{generalized Noisy memory problem}.

The proofs can be carried out by mimicking our proofs from Sections \ref{sec: sufficient}-\ref{sec: NMP}.  In addition to replacing $Z$ by $Z'$ in all of Section \ref{sec: TheProblem}, the only difference from the proofs in Sections \ref{sec: sufficient}-\ref{sec: NMP} is that we need to 
\begin{itemize}
 \item
replace $K$ by $K'$, and $$\int_{t-\delta}^t K(s) dB(s) \quad\textnormal{by}\quad\int_{t-\delta}^t K'(s)\phi(t,s)dB(s)  , $$ with $K'$ satisfying
\small \begin{align}
  d K'(t)&= (K'(t),  K'(t-\delta), \int_{t-\delta}^t  K'(r)\phi(t,r) dB(r),\eta(t))\cdot\Big[\nabla b(t,\mathbf X(t), \pi(t))\, dt \nonumber\\
& +\nabla \sigma(t,\mathbf X(t), \pi(t))\, dB(t)
+\int_{\mathbb{R}_0}\nabla\gamma (t,\mathbf X(t), \pi(t),\zeta) \tilde N(dt,d\zeta)\Big]^\mathsf{T} \tag{\ref{eq:derivativeProcess}'}\\
 K'(t)&=0,\quad t\in[-\delta,0],\nonumber
\end{align}\normalsize
throughout section \ref{sec: NMP}
\item
replace terms of the form
$$\E[D_t\frac{\partial \mathcal H}{\partial z}(s)|\mathcal F_t]\quad\textnormal{by }\quad\E[D_t\frac{\partial \mathcal H}{\partial z}(s)|\mathcal F_t]\phi(t,s)$$ in
\eqref{eq:Z1}-\eqref{eq:Z2}, \eqref{eq: 2.21},\eqref{eq:NMPlastLemma2} and \eqref{eq:NMPlastLemma3}.
\end{itemize}
\begin{example}
Reconsider, the optimal consumption problem from Example \ref{example:SmallerGt}, this time depending on the generalized noisy memory process $Z'(t)$:
\begin{align}
dX(t) &= (Z'(t) + a_1 X(t) - \pi(t)) dt + \sigma(t, X(t), Y(t), Z'(t), \pi(t)) d B(t)  \nonumber\\[\smallskipamount]
&\quad+\int_{\mb{R}} \gamma(t, X(t), Y(t), Z'(t), \pi(t), \zeta) \tilde{N}(dt, d\zeta);  \hspace{1cm} t \in [0,T], \label{eq: exSDE1}\\[\smallskipamount]
X(t) &= \xi(t);  \hspace{6.5cm} t \in [-\delta, 0].\nonumber
\end{align}with
\begin{align*}
Z'(t)=\int_{t-\delta}^t \phi(t, s)X(s)dB(s).
\end{align*}
Here $X(t)$ is a cash flow and $\pi$ is the consumption rate. We let $\phi$ be deterministic.  It is reasonable to choose $\phi(t,\cdot)$ as a function gradually increasing from $0$ at time ${t-\delta}$ to some $a_0\in\mathbb R$ at time $t$, however this is not necessary for the analysis. We leave the performance functional and the set of admissible controls as in Section \ref{sec: example}. Then, the Hamiltonian is
\begin{align*}
&\mathcal H(t, x,y,z, u, p, q, r(\cdot)) \\
&= f(t, u) + ( z + a_1 x - u)p + \sigma(t, x,y,z, u)q+ \int_{\mb{R}} \gamma(t, x,y,z, \zeta)r(\zeta)\nu(d\zeta),
\end{align*}
and the adjoint equation has a deterministic solution satisfying
\begin{align*}
dp(t) &= -\{ a_1 p(t) + \int_t^{t + \delta} \E \Big[ D_t \big( p(s)\big) | \mc{F}_t \Big]\phi(t, s) \boldsymbol{1}_{[0,T]}(s) ds \}dt, \\[\smallskipamount]
p(T) &= 1,
\end{align*} i.e. $p(t) =  e^{a_1(T-t)}, q=0, r=0$, as in Example \ref{example:SmallerGt}. The first order condition for maximality of $\pi^*$ is
\begin{align}\label{ex:sufficientCond}
\E\Big[ \frac{\partial f}{\partial u}(t,\pi^*(t)) | \mc{G}_t\Big] = p(t).
\end{align} We can verify using the sufficient maximum principle (Theorem \ref{thm: sufficient}) that \eqref{ex:sufficientCond} is indeed a sufficient condition for optimality of $\pi^*$. We notice in particular that the dependence on the noisy memory process $Z'(t)$ does not affect our choice of optimal strategy.

\end{example}

\appendix

\nocite{MR754561}
\end{document}